\documentclass[11pt,letterpaper]{article}

\usepackage[left=1in,right=1in,top=1in,bottom=1in]{geometry}
\usepackage{amsmath,amssymb,amsfonts,amsthm}
\usepackage{mathtools}
\usepackage{thmtools}
\usepackage{graphicx}
\usepackage{enumitem}
\usepackage[colorinlistoftodos]{todonotes}

\usepackage{hyperref}
\usepackage[capitalise]{cleveref}
\usepackage{bookmark} 
\usepackage{url}
\usepackage{color}
\usepackage{verbatim}
\usepackage{comment}

\newtheorem{theorem}{Theorem}
\newtheorem{lemma}[theorem]{Lemma}
\newtheorem{corollary}[theorem]{Corollary}
\newtheorem{proposition}[theorem]{Proposition}
\newtheorem{conjecture}[theorem]{Conjecture}

\theoremstyle{definition}
\newtheorem{definition}[theorem]{Definition}

\newtheorem{question}[theorem]{Question}

\newtheorem{fact}[theorem]{Fact}

\newtheorem{claim}[theorem]{Claim}

\newcommand{\R}{\mathbb{R}}
\newcommand{\Z}{\mathbb{Z}}

\newcommand{\Sp}{\mathbb{S}}
\newcommand{\F}{\mathbb{F}}

\newcommand{\RP}{\mathbb{R}\mathrm{P}}
\DeclareMathOperator{\GHD}{GHD}
\DeclareMathOperator{\VR}{VR}

\DeclareMathOperator{\conv}{conv}

\newcommand{\Nrv}{\mathrm{Nrv}}
\newcommand{\poly}{\mathrm{poly}}

\DeclareMathOperator{\srank}{\mathrm{rank}^\pm}
\DeclareMathOperator{\VC}{\mathrm{VC}}

\DeclareMathOperator{\sign}{sign}

\newcommand{\ind}{\operatorname{ind}_{\mathbb{Z}_2}}
\newcommand{\coind}{\operatorname{coind}_{\mathbb{Z}_2}}

\newcommand{\eps}{\varepsilon}

\newcommand{\sd}{\mathrm{sd}}

\begin{document}

\title{A $\mathbb{Z}_2$–Topological Framework for Sign-rank Lower Bounds}
\author{
Florian Frick\thanks{Department of Mathematical Sciences, Carnegie Mellon University. Email: \texttt{frick@cmu.edu}.}
\and
Kaave Hosseini\thanks{Department of Computer Science, University of Rochester. Email: \texttt{kaave.hosseini@rochester.edu}.}
\and 
Aliaksei Vasileuski\thanks{Department of Mathematical Sciences, Carnegie Mellon University. Email: \texttt{avasileu@andrew.cmu.edu}.}
}

\maketitle
\begin{abstract}

We develop a topological framework for proving lower bounds on sign-rank via $\mathbb{Z}_2$–equivariant topology, and use it to resolve the sign-rank of the Gap Hamming Distance problem up to lower-order terms.

For every (partial) sign matrix $A$, we associate a free $\mathbb{Z}_2$–simplicial complex $S(A)$ and show that  sign-rank of $A$ is characterized by the linear analog of $\mathbb{Z}_2$-index of $S(A)$. As a consequence, the classical $\mathbb{Z}_2$–index of $S(A)$ lower bounds the sign-rank of $A$, which reduces sign-rank lower bounds to topological obstructions.
This reduction allows us to use various tools from $\mathbb{Z}_2$–equivariant topology, particularly in regimes where classical lower-bound techniques break down. 

As the main application, we consider the  Gap Hamming Distance function
$\mathrm{GHD}_k^n$ (defined for $k < n/2$), which distinguishes pairs of strings in $\{0,1\}^n$
with Hamming distance at most $k$ from pairs with distance at least $n-k$. 
We prove an essentially tight lower bound and show that for any $k$,
\[ 
\text{sign-rank}(\mathrm{GHD}_k^n) = (1-o_k(1))\,2k. 
\]
 where the $o_k(1)$ term is $O\left(\sqrt{\frac{\log k}{k}}\right)$.
This improves on the previous lower bound of 
Hatami, Hosseini, and Meng (STOC 2023) who proved that sign-rank of 
$\mathrm{GHD}_k^n$ is at least $\Omega(k/\log(n/k))$.

A key technical ingredient is a new analysis of the $\mathbb{Z}_2$-coindex (which lower bounds $\mathbb{Z}_2$-index) of the Vietoris--Rips complex of the hypercube in the sparse regime which yields an essentially tight lower bound. Previously, no results were known in the sparse regime.

\end{abstract}
\newpage
\tableofcontents

\newpage

\section{Introduction}
The sign-rank of a partial matrix~$A$, denoted $\srank(A)$, is the minimum rank of any real matrix~$B$ such that $\sign(B_{ij}) = A_{ij}$ whenever $A_{ij}\in \{-1,+1\}$.
 Sign-rank is a fundamental complexity measure in theoretical computer science and mathematics, yet
despite decades of study, it remains notoriously difficult to analyze. 
There are   three known general lower bound tools for sign-rank, which can be summarized as \emph{VC dimension} \cite{paturi1986probabilistic}, variants of Forster's method based on \emph{spectral norm}/\emph{average margin} \cite{forster2002linear,forster2001relations,forster2006smallest, razborov2010sign,sherstov2011unbounded,linial2007complexity,hatami2022lower}, and \emph{rectangle based} methods \cite{alon2005crossing,hatami2022lower}.
    However, these lower bounds are  effective only when the given matrix contains a large \emph{pseudorandom} piece (i.e., has large VC dimension, or spectral norm is small, or it does not contain large monochromatic rectangles) and there are $N$ by $N$ matrices of sign-rank $\poly(N)$ where these methods cannot produce any $\Omega(1)$ lower bound \cite{hatami2022lower}.
The reader is referred to  \cite{hatami2022lower} for an in-depth discussion of limitations of these methods. 
    This lack of flexible tools to  lower bound sign-rank without requiring pseudorandomness has stalled progress on many fundamental problems about sign-rank. 

 A benchmark problem to study sign-rank is the Gap Hamming Distance problem, which is a partial function   $\GHD_{k}^n\colon \{0,1\}^n\times\{0,1\}^n \to \{-1,1,*\}$   defined for a parameter $k< n/2$ by:
\[
\GHD_{k}^n(x,y)=
\begin{cases}
1 & \text{if } d_H(x,y)\leq k,\\
-1 & \text{if } d_H(x,y)\geq n-k,\\
* & \text{otherwise}.
\end{cases}
\]
Here $d_H(x,y)$ is the Hamming distance between $x$ and $y$. 
Gap Hamming Distance is a fundamental object and  has been intensively studied in different parameter regimes, 
used as a lower bound primitive in randomized communication complexity and streaming algorithms \cite{indyk2003tight,brody2010better, chakrabarti2011optimal,sherstov2012communication,jayram2013optimal}, separating the power of communication models~\cite{papakonstantinou2014overlays,chattopadhyay2019equality,gavinsky2025unambiguous,blondal2025borsuk}, separating partial and total learnability of half-spaces with margin \cite{alon2022theory,chornomaz2025spherical, blondal2025borsuk,chase2024local}, and lower bounds for dimensionality reduction\cite{hatami2023borsuk}. Further more, Gap Hamming Distance  is  the XOR lift of the \emph{approximate majority} function with a broad range of applications in circuit complexity and query complexity, and it is also the discrete analog of the concept class of \emph{half-spaces with margin} which is a canonical object in learning theory.

Next we briefly discuss the sign-rank of $\GHD_k^n$ in different parameter regimes. 
 In the case of odd $n$ and $k=(n-1)/2$, $\GHD_k^n$ is the XOR lift of the exact Majority function, and has VC dimension $n$ which implies $\srank(\GHD_k^n)=n$. 
More generally, the function $\GHD_k^n$ has been studied in two parameter regimes:

\noindent
(1) $k\approx n/2-\sqrt{n}$: in this regime one can get $\Omega(n)$ sign-rank lower bounds using VC dimension or Forster's method as $\GHD_k^n$ still exhibits strong mixing behavior.

\noindent
(2) $k = \eps n$ for constant $\eps\in (0,0.49)$: 
in this regime, it is known that the VC dimension and average margin   are bounded and hence unable to prove super-constant lower bounds on its sign-rank.  However, using  isoperimetric inequalities of hypercube \cite{frankl1981short, frankl1987forbidden}, it is easy to see that $\GHD_k^n$ is still  mildly pseudorandom in the sense that it does not contain large monochromatic rectangles, which implies $\Omega(n)$ sign-rank lower bounds \cite{hatami2023personal}.  
However, once $k$ is  $O(\sqrt{n})$, this flicker of pseudorandom behavior entirely dies down and none of these methods can produce any super-constant lower bounds.

Recently, Hatami, Hosseini, and Meng~\cite{hatami2023borsuk} introduced a  new lower bound on the sign-rank of $\GHD_k^n$ by using a topological argument based on the \emph{Borsuk--Ulam theorem} \footnote{The Borsuk-Ulam theorem~\cite{borsuk1933drei}  states that  
for every continuous map on the $d$-dimensional unit sphere
\(
f \colon \Sp^d \to \mathbb{R}^d
\)
there exists a point $x \in \Sp^d$ such that
\(
f(x) = f(-x).
\)}
and showed  that 
\begin{equation}\label{eq:HHM}
\srank(\GHD_{k }^n) = \Omega\left(\frac{k}{\log(\frac{n}{k})}\right)    
\end{equation}
which is nontrivial for $k = \Omega(\log n)$. 
The idea of \cite{hatami2023borsuk} was to first consider the  analog of $\GHD_{k}^n$ defined on $\Sp^{n-1} \times \Sp^{n-1}$ which is amenable to Borsuk-Ulam theorem,
give a tight lower bound on its sign-rank \footnote{The class of half-spaces with margin $\eps$ is defined as $\mathbb{G}_\eps^n\colon\mathbb{S}^{n-1}\times\mathbb{S}^{n-1}\to\{-1,1,*\}$ by $\mathbb{G}_\eps^n(x,y)=\mathrm{sign}(\langle x,y\rangle)$ if $|\langle x,y\rangle|\ge \eps$, and $*$ otherwise. It was shown in \cite{hatami2023borsuk} that $\srank(\mathbb{G}_\eps^n)=n$}
  , and then apply a randomized reduction to transfer the lower bound to $\GHD_k^n$. 
  
  \medskip
  One limitation of the Borsuk-Ulam approach of \cite{hatami2023borsuk} is that because of the inevitable randomized reduction, it  cannot be improved to handle the regime $k=o(\log n)$.
In this paper, we take a different approach and resolve the sign-rank of $\GHD_k^n$ for any choice of $k < n/2$, up to lower order terms. 
\begin{equation}
    \srank(\GHD_k^n) \geq (1- o_k(1))2k.
\end{equation}
where the $o_k(1)$ term is $O\left(\sqrt{\frac{\log k}{k}}\right)$.  It is easy to see that $\srank(\GHD_{k}^n)\leq 2k+1$ (See \Cref{fact:ghdupper}).
To the best of our knowledge, this is the first finite matrix family whose sign-rank is resolved up to lower order terms. 

  As the main technical ingredient, we give a new analysis of the \emph{$\Z_2$-coindex} of \emph{Vietoris--Rips (VR) complex} of the $n$-hypercube. These are central objects in metric topology which we briefly discuss now.
 For a metric space~$X$ and scale parameter~$k \ge 0$, the Vietoris--Rips complex~$\VR(X,k)$ is the simplicial complex of all finite subsets of~$X$ of \emph{diameter} at most~$k$. Thus this complex provides a combinatorial-topological encoding of pairwise metric proximity. It was first introduced by Vietoris~\cite{Vietoris1927} to define a homology theory for metric spaces, and later used by Rips and Gromov~\cite{Gromov1987} in the context of geometric group theory to study hyperbolic groups. More recently, topological properties of VR complexes have been studied in the computational topology community in the context of persistent homology~\cite{Carlsson2009, EdelsbrunnerHarer2010, ghrist2008barcodes}. Despite their widespread use and popularity, the topology of VR complexes is poorly understood: only partial results are known for VR complexes of various manifolds~\cite{AdamsBushVirk2025ConnectivitySpheres} and of hypercubes, for example. For hypercubes~\cite{AdamaszekAdams2022, AdamsVirk2024, Feng2025, Shukla2023} exact results are only known in the regime $k\le 3$, along with asymptotic results for $k\geq n/2$. Here we prove essentially tight lower bounds for the $\Z_2$-coindex of VR complexes of $n$-hypercubes for all $k < n/2$.
  

\medskip

 
Another limitation of the Borsuk-Ulam argument of \cite{hatami2023borsuk} is that it is highly tailored to the specific structure of $\GHD_k^n$ and it is not clear whether one can use topology to study the sign-rank of arbitrary sign matrices.  This naturally raises the following meta-question:
\begin{center}
\emph{Is there  a systematic topological framework  for sign-rank lower bounds of arbitrary sign matrices?}
\end{center}

In this paper   we build a framework using $\Z_2$-topology to lower bound the sign-rank of any sign matrix. 
First, we introduce a \emph{correspondence} between partial sign matrices and \emph{free $\Z_2$--simplicial complexes}, and reinterpret sign-rank as the linear analogue of \emph{$\Z_2$-index}. This turns lower bounds on sign-rank into questions about equivariant topology.
$\Z_2$-index is a  global topological invariant amenable to a broad range of equivariant methods far 
beyond the Borsuk--Ulam theorem, which will be discussed later in the paper.



After applying our framework to sign-rank of $\GHD_k^n$, we also reinterpret VC dimension topologically and obtain a chain of inequalities between VC dimension, $\Z_2$-coindex, $\Z_2$-index, and sign-rank, along with separations between all these parameters. 

\noindent
(1) \emph{VC vs. $\Z_2$-coindex}: for partial matrices, our result for Gap Hamming already yields a $1$-vs-$\Omega(\log N)$ separation between VC dimension and $\Z_2$-coindex. We get improved separations  by translating state of the art  constructions of small triangulations of the real projective space due to Adiprasito et al. \cite{adiprasito2022subexponential}.  For \emph{total matrices}, we leave the possibility that VC dimension and coindex are \emph{equivalent in the constant regime}, which if true, entirely reframes PAC learnability of total concept classes as a topological property and has  important implications on VC dimension  of disambiguations of Gap Hamming Distance, which we discuss later.

\noindent
(2) \emph{$\Z_2$-index vs. sign-rank}:
we also establish separations between $\Z_2$-index and sign-rank by developing a general method to upper bound $\Z_2$-index.  For total matrices, we leave the possibility that $\Z_2$-index and sign-rank are equivalent in the constant regime, which if true will imply that  \emph{constant sign-rank} is a topological phenomenon and would allow one to reformulate many fundamental open problems about sign-rank  in terms of $\Z_2$-index of $\Z_2$-complexes.

To state our results formally, we briefly introduce the main framework and the notion of \emph{sign complex} of a sign matrix. For more details see \Cref{section:framework}. The reader is referred to \Cref{section:z2background} for all the necessary background on $\Z_2$-equivariant topology.

\subsection{The main framework} 
    Given a matrix $A\in \{-1,1,*\}^{M\times N}$ with columns indexed by $[N]$, we  construct an  abstract simplicial complex called the \emph{sign complex}  denoted by $S(A)$. This is a simplicial complex built on two copies of the column set, namely
\[
\{1^-,1^+,\dots,N^-, N^+\}.
\]
For each row $r \in \{-1,1,*\}^N$, define the simplex
\[
\sigma_r^+ \coloneqq \{\, i^+ : r_i = +1 \,\} \cup \{\, i^- : r_i = -1 \,\},
\]
and its antipodal copy
\[
\sigma_r^- \coloneqq \{\, i^- : r_i = +1 \,\} \cup \{\, i^+ : r_i = -1 \,\}.
\]
The simplices of $S(A)$ are all subsets of $\sigma_r^+$ or $\sigma_r^-$ for some row $r$ of $A$.
There is a natural free $\Z_2$-action on $S(A)$ that exchanges $i^-$ and $i^+$. 
 This  turns $S(A)$ to a  \emph{free $\mathbb Z_2$--simplicial complex}.
For more detail, see \Cref{section:framework}.
We then consider a geometric realization $|S(A)|\subset \R^N$ by exploiting its antipodal symmetry given by
\[
j^+ \mapsto e_j,\qquad j^- \mapsto -e_j
\]
where $e_j$ is the $j$th standard basis element 
in \(\R^N\), and extending affinely on each simplex. 
We show the sign-rank of $A$ is one less than the smallest dimension $d$ such that $|S(A)|$  can be mapped linearly into $\R^{d+1}$ while avoiding the origin.
 More formally, for a free $\Z_2$-complex $K$, let $\ind^{\mathrm{lin}}(K)$ denote the smallest integer $d$ such that there exists a linear map $g\colon \R^N \to \R^{d+1}$ such that 
 \[
0 \notin g(|K|).
\] 
Such a linear map $g$ induces a simplex-wise linear $\Z_2$-equivariant map $|K| \to \R^{d+1}$ and every simplex-wise linear $\Z_2$-equivariant map $|K| \to \R^{d+1}$ comes about in this way.  We prove the following.
\begin{restatable}[Main lemma 1]{lemma}{mainlem}
\label{lem:main-lemma}
For every partial sign matrix \(A\),
\[\ind^{\mathrm{lin}}(S(A)) = \srank(A)-1.\]
\end{restatable} 
See \Cref{section:indlinsignrank} for details and proof.
In particular, sign-rank can be viewed as a linear analogue of a global topological invariant called the \emph{$\Z_2$-index}. The $\Z_2$-index of a free $\Z_2$-complex $K$, denoted by $\ind(K)$, is the smallest integer $d$ for which there exists a $\Z_2$–equivariant continuous map (i.e., continuous maps that respect the antipodal symmetry $f(-x)=-f(x)$)
\[
f\colon|K| \to \Sp^d.
\]
Since the map $g$ induces a $\Z_2$–equivariant map $f\colon |K|\to \Sp^{d}$, defined by $f(x)\coloneq\frac{g(x)}{\|g(x)\|}$, we have that for any free $\Z_2$-complex $K$
\[
\ind(K) \le \ind^{\mathrm{lin}}(K)
\]
which combined with \Cref{lem:main-lemma} implies the following inequality.
\begin{restatable}[Main inequality]{corollary}{mainineq}\label{corollary:mainineq}
For every partial sign matrix $A$,
\[
\ind(S(A)) \le \srank(A)-1.
\]
\end{restatable}
\Cref{corollary:mainineq} allows one to use the full machinery of $\Z_2$-equivariant topology to prove sign-rank lower bounds. Index lower bound techniques include Borsuk--Ulam type arguments
\cite{borsuk1933drei,lovasz1978kneser,matouvsek2003using}, nerve lemmas and connectivity \cite{matouvsek2003using},
the Fadell--Husseini cohomological index / Stiefel--Whitney classes~\cite{fadell1988ideal,blagojevic2017beyond, matouvsek2003using},
equivariant obstruction theory
\cite{blagojevic2017beyond},
and discrete Morse theory
\cite{forman1998morse}.
In particular, our framework allows one to  apply topological methods in settings where classical sign-rank techniques are ineffective.  

Moreover, the correspondence between free $\Z_2$-simplicial complexes and partial sign-matrices is in fact two-way:
\emph{ every free $\Z_2$–simplicial complex $K$ is isomorphic to $S(A_K)$ for some partial sign matrix $A_K$. If $K$ has $2N$ vertices and $M$ facets, then $A_K$ has $N$ columns and $M$ rows.
Moreover, $A_K$ is a total sign matrix iff all facets of $K$ have size $N$.}
Hence, in order to upper bound the index of  $K$, it is enough to upper bound the sign-rank of $A_K$. This viewpoint can be  beneficial in settings where topology of $K$ is difficult to study but $A_K$ is amenable to sign-rank \emph{upper bound} methods including sign alternating number \cite{alon1985geometrical}, polynomial threshold degree \cite{klivans2001learning}, support rank \cite{goos2025sign},  communication protocols \cite{paturi1986probabilistic}, protocols with Hamming oracles\cite{hatami2022lower, goos2025sign}, and analytic parameters such as factorization norm \cite{hatami2022lower}.


\medskip
As our main application, we combine several tools from $\Z_2$-topology to obtain an essentially tight lower bound on the sign-rank of $\GHD_k^n$ for all $k < n/2$ which is discussed next.

\subsection{Main application: sign-rank of Gap Hamming Distance}\label{subsection:gap_hamming_intro}

It is easy to see that, by projecting to the first $2k+1$ coordinates:
\[\srank(\GHD_k^n)\leq 2k+1.\] 
Indeed, identify rows and columns with $\{\pm1\}^n$ and map each
$x,y\in\{\pm1\}^n$ into $x',y'\in\mathbb{R}^{2k+1}$ by only keeping the first $2k+1$
coordinates and observe that $\sign(\langle x',y'\rangle) = 1$ if $d_H(x,y)\leq k$ and $-1$ if $d_H(x,y)\geq n-k$.
See \Cref{fact:ghdupper} for details.
In this paper we show that this projection map is  optimal up to lower order terms.
\begin{restatable}[Main application]{theorem}{strongbound}\label{thm:strongbound}
For any $k < n/2$,
\[
\srank(\GHD_{k}^n)\ge (1-o_k(1))\cdot 2k .
\]
where the $o_k(1)$ term is $O\left(\sqrt{\frac{\log k}{k}}\right)$.
\end{restatable}
See \Cref{section:application_gaphamming} for the proof.
\noindent
As mentioned before, the best previous bound, due to Hatami, Hosseini, and Meng \cite{hatami2023borsuk} was that
\begin{equation}\label{eq:HHM}
\srank(\GHD_{k }^n) = \Omega\left(\frac{k}{\log(\frac{n}{k})}\right).    
\end{equation}

\Cref{thm:strongbound} is proved by lower bounding the \emph{$\Z_2$-coindex} which can be viewed as a dual of $\Z_2$-index. The $\Z_2$-coindex of a $\Z_2$-complex $K$ is the largest $d$ for which there is a $\Z_2$-equivariant continuous map 
\[
\Sp^d\to |K|.
\]
Moreover, the definitions imply that 
\begin{equation}
\coind(K)\leq \ind(K).
\end{equation}
An instructive analogy is to compare the coindex and index with \emph{clique number} and \emph{chromatic number} of graphs. While clique number (i.e., homomorphism from largest clique to the graph) is like coindex and detects large forced local structure, but the chromatic number (i.e., smallest homomorphism to a clique) is like index and captures global obstructions and can be significantly higher than clique number.  

Nevertheless, in the case of $S(\GHD_k^n)$ we prove that coindex and index are essentially equal and large, which implies the bound on sign-rank.
\begin{restatable}{theorem}{strongboundGHD}\label{thm:strongboundGHD}
For any \(k< n/2\),
\[
\coind\!\bigl(S(\GHD_k^n)\bigr)\ge (1-o_k(1))\cdot 2k .
\]
\end{restatable}

In particular, the same lower bound holds for the sign-rank of~$\GHD_k^n$. In \Cref{section:weaker} we give a weaker lower bound of $k$ on the $\Z_2$-coindex of $S(\GHD_{k}^n)$ by exploiting
the topological connectivity of the $k$-skeleton of the cubical
\emph{CW-complex} structure on the boundary of the hypercube.
While this approach does not give the optimal bound and is shy of a factor of 2,
it shows the topological connectivity mechanism in a more straightforward way.

We prove \Cref{thm:strongboundGHD} by exploiting more than cubical faces of the hypercube. 
The right object to study is called the Vietoris-Rips complex of the hypercube, denoted by $\VR(Q_n,k)$. 
The Vietoris--Rips complex of the hypercube $\VR(Q_n,k)$ is the simplicial complex on
vertex set $\{0,1\}^n$ whose faces are subsets of vertices of
pairwise Hamming distance at most $k$. 
As discussed previously, very little is known about the topology of $\VR(Q_n,k)$ beyond $k\leq 3$ or the dense regime $k > n/2$, and essentially no sharp topological information is available in the sparse regime $k \ll n/2$.

As our main technical lemma, we prove a sharp lower bound on coindex of $\VR(Q_n,k)$ and then give a $\Z_2$-equivariant map from $\VR(Q_n,k)$ to $S(\GHD_k^n)$ which implies \Cref{thm:strongboundGHD}.
\begin{restatable}[Main lemma 2]{lemma}{VRZtwo}\label{lem:VRZ2}
For any $k < \frac{n}{2}$,
\[
\coind(\VR(Q_n,k))\geq (1-o_k(1))\,2k.
\]
\end{restatable}

\paragraph{Proof idea of \Cref{lem:VRZ2}.}
A standard way to lower bound $\coind(\VR(Q_n,k))$ is to prove
a lower bound on the homotopy connectivity of $\VR(Q_n,k)$. However, this does not work
in our setting, since $k$ may be arbitrarily small compared to $n$, making the
complex extremely sparse and poorly connected.
Instead, the key step is to identify a \emph{suitable symmetric subcomplex $K$} of
$\VR(Q_n,k)$ and prove that its connectivity is nearly $2k$ which implies a lower bound on $\coind(K)$ and then by monotonicity of $\coind(\cdot)$ we get the lower bound on $\coind(\VR(Q_n,k))$. 
To lower bound connectivity of $K$, the idea is to
cover it by pieces supported on lower-dimensional faces of the
hypercube, each isomorphic to $\VR(Q_t,k)$ for $t\approx 2k-\sqrt{k}$. A result of
Bendersky and  Grbic \cite{bendersky2023connectivity} implies that each such piece has high
connectivity. We then glue these pieces together using two applications of the nerve lemma,
which allows us to transfer this local connectivity to the  connectivity of~$K$, and finally transfer this to a lower bound on  the $\Z_2$-coindex of $\VR(Q_n,k)$. 


\subsection{VC dimension and coindex}
The $VC$ dimension of a matrix $A\in \{-1,1,*\}^{M\times N}$ is the largest number $d$
for which there exists a set of columns $S\subseteq[N]$ with $|S|=d$
such that every sign pattern $P\in\{-1,1\}^S$ appears as the restriction
of some row of $A$ to $S$. 
It is well known that by Radon's theorem, $VC(A)$ provides a lower bound on sign-rank.
However, this bound is typically weak as VC dimension only captures local combinatorial structure: there exist matrices with
constant VC dimension whose sign-rank is polynomially large
\cite{alon2016sign}.

Here, VC dimension of our sign complex framework has a natural topological interpretation. 
Recall that the vertex set of the sign complex
$S(A)$ is
\[
\{1^+,1^-,\dots,N^+,N^-\},
\]
and that $S(A)$ is a subcomplex of the boundary of the crosspolytope
$\partial \Diamond_N$. Let $\omega_\Diamond(K)$ of a free $\Z_2$-complex be the 
 largest $d$ for which the complex $K$ has $\partial \Diamond_d$ as an induced  free $\mathbb{Z}_2$--subcomplex. 
 Then it is easy to see that 
 \[\frac{\omega_\Diamond(S(A))}{2} \leq \VC(A)\leq \omega_\Diamond(S(A))\]
 See \Cref{thm:VCDiamond}.
Since $\partial \Diamond_d$ is 
homeomorphic to the sphere $\Sp^{d-1}$ with the antipodal action,
its presence inside $S(A)$ immediately implies a lower bound on
the $\mathbb{Z}_2$--coindex, namely, the largest $k$ for which there is a continuous $\Z_2$-equivariant map from $\Sp^k$ to $|S(A)|$.
\begin{proposition}\label{prop:vccoindex}
For every partial sign matrix $A$,
\[
\VC(A)-1 \le \coind(S(A)).
\]
\end{proposition}


Overall, we obtain the following chain of inequalities:
\begin{equation}\label{eq:inequalitychain}
\VC(A)-1
\;\le\;
\coind(S(A))
\;\le\;
\ind(S(A))
\;\le\;
\srank(A)-1.
\end{equation}
Hence the classical VC lower bound on sign-rank appears in our
framework as a very special case of a topological obstruction:
it corresponds (up to a factor of 2) to the existence of a large antipodal
crosspolytope embedded inside the sign complex. 
However, the topological viewpoint is considerably more flexible.
While VC dimension only detects the presence of a crosspolytope
boundary inside $S(A)$, the $\mathbb{Z}_2$--index can capture much
more complicated global topology.  This additional power is crucial
for obtaining strong lower bounds for sign-rank in settings where VC dimension
remains constant. 

A similar inequality as in \Cref{prop:vccoindex} has also been independently introduced by \cite{chornomaz2025spherical} in the context of learning theory where they associate a  simplicial complex similar to the sign complex, to a total concept class and relate its coindex to VC dimension, list replicability, and many other important learnability parameters. Also in \cite{blondal2025simplicial}, coindex has been proven to lower bound \emph{simplicial covering dimension} which is another fundamental parameter characterizing list replicability.  Both of these works have led to powerful applications in learning theory. 

Nevertheless, even though coindex is useful for lower bounding fundamental learning parameters, but $\Z_2$-index captures more \emph{global obstructions} and can give stronger sign-rank lower bounds. 
In  \Cref{section:indexcoindex_separation} we use classical constructions in topology that give sign complexes with coindex equal to 1, but index arbitrarily large. This is done via another topological invariant called \emph{Stiefel-Whitney height} that sits between coindex and index.



\subsection{Separation results}
\medskip
A basic remaining question we need to understand is to what extent the complexity measures in the chain of inequalities in \Cref{eq:inequalitychain} can be separated. 

\paragraph{Separating $\Z_2$-index and sign-rank.}
We split to the case of total matrices and the more general case of  partial matrices.

\begin{enumerate}
    \item \textbf{Total matrices.}
    We show that a random $N\times N$ total sign matrix $A$ has
    $\ind(S(A)) = O(\log N)$ with high probability,
    while its sign-rank is $\Omega(N)$ \cite{alon1985geometrical}.
    Similarly we show that for $N=2^n$, the $N$ by $N$ Hadamard matrix $H_N$ has $\ind(S(H_N))$ at most $ 2\log N$. However, it is already known, from seminal work of Forster \cite{forster2002linear} that  its sign-rank is at least $\sqrt{N}$.

      More generally, we introduce a combinatorial parameter, for an arbitrary free $\Z_2$-complex~$K$, as the \emph{height} of its associated partial sign matrix $A_K$, and show that $\ind(K)$ is at most twice its height.  This provides a systematic way to obtain logarithmic upper bounds
    on $\Z_2$-index even when the underlying linear dimension is large.
    For more details, see \Cref{section:index-signrank-total-separation}. 
    
    \item \textbf{Partial matrices.} We construct random  $N\times N$ partial matrices built on top of the incidence matrix of the finite projective plane that has $\Z_2$-index equal to 1, but has sign-rank $\Omega(\sqrt{N})$. For more details, see \Cref{section:index_signrank_separation_partial}.
\end{enumerate}

We leave the following remaining important question as open.
\begin{restatable}{question}{TotalSeparationQuestion}\label{question:totalseparation}
Does there exist a function $f\colon \mathbb{N}\to\mathbb{N}$ such that for every total sign matrix~$A$,
\[
\srank(A)\leq f(\ind(S(A))) ?
\]
\end{restatable}
If the answer to \Cref{question:totalseparation} is positive, then theories of $\Z_2$-index and sign-rank for total matrices will be equivalent in the $O(1)$ regime. This will have major consequences, as one can reformulate a large number of fundamental open problems related to sign-rank entirely in the language of topology. Some of these  problems will be discussed in conclusion in \Cref{section:conclusion}.

\paragraph{Separating  VC dimension and coindex.}
Next, we address the question of separating $VC$ and coindex.
We obtain the following separation using a result of Adiprasito et al.~\cite{adiprasito2022subexponential} on triangulations of the projective space with few vertices. For details see \Cref{section:vccoindex_separation}.
\begin{restatable}[Separating VC and coindex]{proposition}{VCcoindexsep}
\label{prop:VC-coindex-separation}
There exists a partial matrix 
$A\in \{-1,1,*\}^{M\times N}$ such that \(\mathrm{VC}(A) = 1\) but
\[
\coind(S(A)) = \Omega\left(\frac{\log^2 N}{\log \log N}\right).
\]
\end{restatable}

However, in the case of total matrices we leave the possibility that coindex and VC dimension are qualitatively equivalent which if true, will entirely reframe learnability of total concept classes (constant VC) as a topological property of having bounded coindex. 
\begin{restatable}{question}{VCcoindexseptotal}
\label{question:VC-coindex-separation-total}
Does there exist a function $f:\mathbb{N}\to\mathbb{N}$ such that for every total sign matrix $A$,
\begin{equation}\label{eq:coindbound}
\coind(S(A)) \le f(\VC(A))?
\end{equation}
\end{restatable}
We will give another interesting implication of \Cref{question:VC-coindex-separation-total} in the concluding remarks in \Cref{section:conclusion} on the VC dimension of disambiguations of Gap Hamming Distance to total functions.

  \subsection{Paper organization}
In \Cref{section:background} we review the
necessary background on sign-rank and
$\mathbb{Z}_2$-equivariant topology.
In \Cref{section:framework} we introduce the main framework.
In \Cref{section:application_gaphamming} we show  the lower bound for the Gap Hamming Distance matrix. In \Cref{section:basic} we discuss VC dimension and other basic properties of the sign complex.
In \Cref{section:separation} we show  quantitative gaps between
VC dimension, coindex, index and sign-rank.
Finally, in \Cref{section:conclusion} we conclude with some open problems.

\section{Background}\label{section:background}
In this section we introduce the necessary background on sign-rank, Gap Hamming Distance, and $\Z_2$-topology.

\subsection{Sign-rank}\label{section:sign-rankbackground}
The \emph{sign-rank} of a total sign matrix $A_{M\times N}\in\{-1,1\}^{M\times N}$, is the minimum rank of a real matrix $B_{M\times N}$ such that 
\[
\mathrm{sign}(B_{x,y}) = A_{x,y}
\qquad \text{for all } (x,y) \in [M]\times [N].
\]
This definition extends naturally to partial sign matrices: if $A_{x,y} = *$, then the corresponding entry $B_{x,y}$ may be any real number.
In other words, the sign-rank of $A\in\{-1,1,*\}^{M\times N}$, denoted by $\srank(A)$, is the smallest integer $d$ for which
there exist vectors
\[
u_1,\dots,u_m \in \R^{d}\setminus\{0\}
\qquad\text{and}\qquad
v_1,\dots,v_n \in \R^{d}\setminus\{0\}
\]
such that for every specified entry $A_{ij}\in\{\pm1\}$ we require
\[
\sign\big(\langle u_i, v_j\rangle\big) \;=\; A_{ij}.
\]
Equivalently, each row vector $u_i$ defines the homogeneous hyperplane
\[
H_i := \{x\in \R^d : \langle u_i,x\rangle = 0\},
\]
and it strictly separates the column-point sets
\[
C_i^+ := \{v_j : A_{ij}=+1\},
\qquad
C_i^- := \{v_j : A_{ij}=-1\},
\]
in the sense that $\langle u_i, v\rangle>0$ for all $v\in C_i^+$ and
$\langle u_i, v\rangle<0$ for all $v\in C_i^-$.
This viewpoint is used in learning theory, where sign-rank is known as \emph{dimension complexity}, capturing the smallest dimension $d$ in which a concept class can be realized as a family of half-spaces and points in $\R^d$.

\subsection{Background on $\mathbb{Z}_2$-Topology}\label{section:z2background}
Here we discuss a minimal background on $\Z_2$-topology that we need to use, and refer the reader to \cite{matouvsek2003using,kozlov2008combinatorial,de2013course} for more details.
 \paragraph{Abstract simplicial complex.}
An \emph{abstract simplicial complex} $K$ on a finite vertex set $V$ is a family 
$K \subseteq 2^V$ such that:
\begin{enumerate}
    \item $\{v\} \in K$ for every $v \in V$,
    \item if $\sigma \in K$ and $\tau \subseteq \sigma$, then $\tau \in K$.
\end{enumerate}
Elements $\sigma \in K$ are called \emph{simplices}. 
The dimension of $\sigma$ is $\dim(\sigma) = |\sigma|-1$, and 
$\dim(K)$ is the maximum dimension of a simplex in $K$. A \emph{facet} of $K$ is a maximal simplex under inclusion.

As a basic example, the \emph{full $d$-dimensional simplex} is the complex
\[
\Delta^d \coloneqq 2^{[d+1]},
\]
i.e., it consists of all subsets of a $(d+1)$-element vertex set.

The \emph{$k$-skeleton} of  $K$ is the subcomplex
\[
K^{(k)} \coloneq \{ \sigma \in K : \dim(\sigma) \le k \}.
\]

Let $K$ and $L$ be simplicial complexes on vertex sets $V(K)$ and $V(L)$.
A map on vertices $\varphi\colon V(K) \to V(L)$ induces  a \emph{simplicial map} $\varphi\colon K\to L$
if for every simplex $\sigma \in K$ the image
\(
\varphi(\sigma) \coloneq \{ \varphi(v) : v \in \sigma \}
\)
is a simplex of $L$ (i.e.\ $\varphi(\sigma) \in L$).
The simplicial bijection $\varphi$ is called an \emph{isomorphism} if
\[
\sigma \in K
\quad\Longleftrightarrow\quad
\varphi(\sigma) \in L.
\]
In which case, we say that $K$ and $L$ are
\emph{isomorphic} and write
\(
K \cong L.
\)

A \emph{$\mathbb{Z}_2$-action} on $K$ is a simplicial map $\tau \colon K \to K$
such that $\tau^2 = \mathrm{id}$ (equivalently, $\tau(\tau(v))=v$ for all vertices $v$).
Furthermore, the map $\tau$ is \emph{free} if $\tau(\sigma)\neq \sigma$ for any simplex $\sigma\in K$.
A free $\Z_2$-action is also called an \emph{involution}.

Let $(K,\tau_K)$ and $(L,\tau_L)$ be simplicial complexes equipped with free
$\mathbb{Z}_2$-actions.
A simplicial map $f \colon K \to L$ is called \emph{$\mathbb{Z}_2$-equivariant}
if it commutes with the involutions, that is,
\[
f(\tau_K(v)) = \tau_L(f(v))
\quad \text{for every vertex } v \in V(K).
\]
Equivalently,
\[
f \circ \tau_K = \tau_L \circ f
\]
as maps of simplicial complexes.

Let $K$ be a simplicial complex with an involution $\tau\colon K \to K$.
It follows that no simplex of $K$ contains both $v$ and $\tau(v)$ for any vertex $v$, hence $K$ can be viewed as a subcomplex of the boundary of the crosspolytope which is discussed next.

Given simplicial complexes $K$ and $L$ on disjoint vertex sets, 
their \emph{join} $K * L$ is defined as
\[
K * L \coloneqq 
\{\sigma \cup \tau : \sigma \in K,\ \tau \in L\}.
\]
If $K$ and $L$ do not necessarily have disjoint vertex sets,
their \emph{deleted join} is
\[
K *_\Delta L
\coloneqq
\{ (\sigma \times \{1\}) \cup (\tau \times \{2\}) : \sigma \in K,\ \tau \in L,\ 
\text{and } \sigma\cap \tau =\varnothing \}.
\]

\paragraph{The boundary of the crosspolytope.}

The $n$-dimensional crosspolytope $\Diamond_d$ is the convex hull of 
$\{\pm e_1,\dots,\pm e_d\} \subset \mathbb{R}^d$.  
Its boundary complex, denoted by $\partial \Diamond_d$ is viewed here as an abstract simplicial complex: the vertices are
\[
\{1^+,1^-,\dots,d^+,d^-\},
\]
and a subset forms a simplex if it does not contain both $i^+$ and $i^-$ 
for any $i$.
The involution exchanging $i^+$ and $i^-$ defines a free $\mathbb{Z}_2$-action.  
Hence $\partial \Diamond_d$ is a free $\mathbb{Z}_2$-simplicial 
complex.
Equivalently, $\partial \Diamond_d$ can be expressed as the deleted join of two copies of the simplex:
\[
\partial \Diamond_d \cong \Delta^{d-1} *_\Delta \Delta^{d-1}.
\]

Finally, any free $\Z_2$-complex can be viewed as a subcomplex of the crosspolytope.
\begin{lemma}\label{lem:subcrosspolytope}
Let $K$ be a simplicial complex on $2N$ vertices with a $\mathbb{Z}_2$-action $\tau$ that is free. Then $K$ is (isomorphic to) a subcomplex of $\partial \Diamond_N$.
\end{lemma}

\begin{proof}
Since the action is free, the vertices of $K$ split into disjoint orbits of size $2$; choose one
representative $v_i$ from each orbit
so that the vertex set of $K$
decomposes as
\[
V(K)=\{v_1,\tau(v_1),\dots,v_N,\tau(v_N)\}.
\]
Denote the vertices of  $\partial \Diamond_N$  with $\{1^+,1^-,\dots,N^+,N^-\}$. 
Define a map $\phi:V(K)\to \{1^+,1^-,\dots,N^+,N^-\}$ by
\[
\phi(v_i)=i^+,\qquad \phi(\tau(v_i))=i^-
\]
and note that $\phi$ 
extends to an injective and $\Z_2$ equivariant map $\phi:K\hookrightarrow \partial \Diamond_N$.
 
It remains to check that $\phi$ extends to a simplicial map.
Let $\sigma\in K$ be any simplex. Strong freeness implies
 $\phi(\sigma)$ contains at most one of $\{i^+,i^-\}$ for each
$i$, hence $\phi(\sigma)$ is a simplex of $\partial \Diamond_N$.
\end{proof}

\paragraph{Geometric realization in the crosspolytope.}
Let $K$ be a free $\Z_2$ simplicial complex on $2N$ vertices. To talk about various topological properties of $K$ such as connectivity, homotopy,  and $\Z_2$-index, we pass to a geometric realization of $K$, denoted by $|K|$. 
By \Cref{lem:subcrosspolytope}, we may identify $K$ with a subcomplex of 
$\partial \Diamond_N$ with vertex set
\[
\{1^+,1^-,\dots,N^+,N^-\}.
\]
We realize $K$ geometrically inside $\R^N$.
Let $\phi : V(K) \to \R^N$ be the embedding defined by
\[
\phi(i^+) = e_i,
\qquad
\phi(i^-) = -e_i
\]
where $\{e_1,\dots,e_N\}$ is the standard basis of $\R^N$.
For each simplex $\sigma \in K$, define its geometric realization as
\[
|\sigma|
=
\left\{
\sum_{v \in \sigma} \lambda_v \, \phi(v)
:
\lambda_v \ge 0,\;
\sum_{v \in \sigma} \lambda_v = 1
\right\}.
\]
The geometric realization of $K$ is
\[
|K|
=
\bigcup_{\sigma \in K} |\sigma|
\subseteq |\partial \Diamond_N| \subset \R^N.
\]
The involution $\tau$ exchanging $i^+$ and $i^-$ corresponds to the map
\[
\tau(x) = -x,
\]
and hence induces a free continuous $\mathbb{Z}_2$-action on $|K|$.

\paragraph{$\Z_2$-index and coindex.}
    Let $K$ be a free $\mathbb{Z}_2$-complex.  
The \emph{$\mathbb{Z}_2$-index} of $K$ is defined as
\[
\mathrm{ind}_{\mathbb{Z}_2}(K)
\coloneqq 
\min \{ d \ge 0 : \exists\ \text{$\mathbb{Z}_2$-equivariant map }
|K| \to \Sp^d \}.
\]
Here $\Sp^d$ is equipped with the antipodal action $x \mapsto -x$.
In particular, if $K$ admits no equivariant map to $\Sp^{d-1}$, 
then $\mathrm{ind}_{\mathbb{Z}_2}(K) \ge d$.
There is also the dual notion of \emph{$\Z_2$-coindex}:
\[
\coind(K):=\max\bigl\{d\ge 0:\ \exists\ \text{a $\Z_2$-equivariant map }
\Sp^d\to |K|\bigr\}.
\]
It is immediate from the definitions that 
\[\coind(K)\leq \ind(K).\]

\paragraph{Connectivity.}

A  standard method to provide a lower bound for $\Z_2$-index of a topological space $X$ is to show that $X$ is highly connected.
A topological space $X$ is \emph{$r$-connected} if for every $k\le r$, every continuous map
\[f\colon \Sp^k \longrightarrow X\]
admits a continuous extension 
\[\tilde{f}\colon B^{k+1} \longrightarrow X\]
on the unit ball~$B^{k+1}$.
The following classical fact relates connectivity with $\Z_2$-coindex.
\begin{proposition}[\cite{matouvsek2003using}, Proposition 5.3.2]
\label{prop:coindex_vs_connectivity}
    Let $X$ be a $r$-connected $\Z_2$-space. Then $\coind(X)\ge (r+1)$. 
\end{proposition}
Moreover, if it is possible to equivariantly map such $X$ to some other space $Y$ then the bound carries over as well. This is particularly useful in a scenario where $Y$ is a $\Z_2$-space of low or unknown (global) connectivity yet one can identify a highly connected symmetric subspace of $Y$.
\begin{corollary}
\label{cor:inder_and_connectivity}
     Let $Y$ be a $\Z_2$-space. Suppose that there exists a $r$-connected $\Z_2$-space $X$ admitting $\Z_2$-equivariant map $f\colon X\to Y$. Then $\ind(Y)\ge (r+1)$.
\end{corollary}  

\paragraph{Covering and nerve lemmas.} An important approach to establish connectivity of a space is to cover it via highly connected pieces and then apply a nerve lemma.

Let $X = |K|$ be the geometric realization of a simplicial complex $K$. A family of subcomplexes \(\mathcal{C} = \{C_i\}_{i\in I}\) is called a \emph{cover} of $X$ if
\[
X = \bigcup_{i\in I}C_i.
\]
Furthermore, $\mathcal{C}$ is called a \emph{good cover} if for every nonempty finite intersection 
\[C_{i_1}\cap \dots\cap C_{i_s}\]
the intersection is contractible.
The \emph{nerve} of $\mathcal{C}$, denoted by $\Nrv(\mathcal{C})$, is the simplicial complex with vertex set $I$ where a finite subset $\{i_1,\cdots,i_s\}\subset I$ spans a simplex if and only if 
\[C_{i_1}\cap \dots\cap C_{i_s}\neq \varnothing.\]

\begin{lemma}[Nerve Lemma]\label{lem:goodcover_nerve}
    Let $X = |K|$ be the geometric realization of a simplicial complex $K$ and let \(\mathcal{C} = \{C_i\}_{i\in I}\) be a finite good cover of $X$ by subcomplexes. Then the geometric realization of $\Nrv(\mathcal{C})$ is homotopy equivalent to $X$:
    \(
    |\Nrv(\mathcal{C})|\simeq X.
    \)
    In particular $|\Nrv(\mathcal{C})|$ and $X$ have the same connectivity.
\end{lemma}

Recall that an $N$-dimensional \emph{CW-complex} is defined inductively starting from a discrete space~$X^{(0)}$, the $0$-skeleton or vertex set, and for $n \le N$ constructing the $n$-skeleton $X^{(n)}$ from $X^{(n-1)}$ by gluing a disjoint union of $n$-dimensional balls into $X^{(n-1)}$ via continuous maps from their boundaries to~$X^{(n-1)}$.
The following more general nerve lemma is useful for us~\cite{bjorner2003nerves}.
 
\begin{lemma}\label{lem:connectivity_nerve}
    Let \(\mathcal{C} = \{C_i\}_{i\in I}\) be a finite cover of a CW-complex $X$ by subcomplexes $C_i$. Suppose there is $k$ such that
    \begin{enumerate}
        \item The nerve $\Nrv(\mathcal{C})$ is $k$-connected.
        \item Every nonempty finite intersection $C_{i_1}\cap\cdots\cap C_{i_s}$ is $(k-s+1)$-connected.
    \end{enumerate}
    Then $X$ is $k$-connected.
\end{lemma}


\paragraph{Order Complex and Barycentric subdivision.}

Let $(P,\le)$ be a finite poset.  
The \emph{order complex} of $P$, denoted $\Delta(P)$, is the simplicial complex
 whose simplices are the finite chains in $P$:
\[
\Delta(P)
=
\big\{
\{x_0,\dots,x_k\} \subseteq P
:
x_0 < \cdots < x_k
\big\}.
\]
Vertices correspond to elements of $P$, and simplices correspond to strictly
increasing chains.

\medskip

For a simplicial complex $K$, let $\mathcal{F}(K)$ be its face poset
(nonempty simplices ordered by inclusion).
The \emph{barycentric subdivision} of $K$ is
\[
\sd(K) \coloneqq \Delta(\mathcal{F}(K)).
\]
Thus vertices of $\sd(K)$ are simplices of $K$,
and a simplex of $\sd(K)$ is a chain
\[
\sigma_0 \subsetneq \sigma_1 \subsetneq \cdots \subsetneq \sigma_k.
\]
There is a canonical homeomorphism
\[
|\sd(K)| \cong |K|,
\]
so barycentric subdivision preserves homotopy type.
If $K$ carries a simplicial $\Z_2$-action,
then the induced action on $\mathcal{F}(K)$ makes $\sd(K)$
a simplicial $\Z_2$-complex.
Freeness is preserved.




\section{The sign complex framework}\label{section:framework}
In this section we introduce the main  framework to assign a free $\Z_2$ simplicial complex $S(A)$ to a given partial sign matrix $A$ and vice versa. Then we show that sign-rank is equal to a linear analog of $\Z_2$-index, and in particular $\ind(S(A))$ is a lower bound on $\srank(A)$.
\subsection{Sign complex}

Let $A\in \{-1,1,*\}^{M\times N}$ be a partial sign matrix.
For each row $r \in \{-1,1,*\}^n$, define the simplex
\[
\sigma_r^+ \coloneqq \{\, i^+ : r_i = +1 \,\} \cup \{\, i^- : r_i = -1 \,\},
\]
and its antipodal copy
\[
\sigma_r^- \coloneqq \{\, i^- : r_i = +1 \,\} \cup \{\, i^+ : r_i = -1 \,\}.
\]

The \emph{sign complex} $S(A)$ is the simplicial complex on vertex set
\[
\{1^+,1^-,\dots,N^+,N^-\}
\]
whose simplices are all subsets of $\sigma_r^+$ or $\sigma_r^-$ for some row $r$ of $A$.
The involution exchanging $i^+$ and $i^-$ defines a free $\Z_2$-action on $S(A)$.
This action turns 
$S(A)$ into a \emph{finite free $\Z_2$-simplicial complex}.
This is because by construction, no face contains both $j^+$ and $j^-$.

\begin{lemma}[Universality of sign complexes]\label{lem:signcomp_universal}
Every finite free $\Z_2$-simplicial complex $K$
is isomorphic to $S(A_K)$ for some partial sign matrix $A_K$.   
\end{lemma}
\begin{proof}
Let $\tau$ be the involution on $K$, hence vertices come in disjoint pairs $\{v,\tau(v)\}$. Index these pairs by $[N]$, labeling them $j^+,j^-$.
Let $\mathcal{F}$ be the set facets of $K$. Note that if $F\in\mathcal{F}$, then $\tau(F)\in \mathcal{F}$. Hence facets come in $\tau$-pairs. Pick an arbitrary facet from each orbit of this action and let $F_1,\cdots,F_m$ be such representatives.

Define a matrix $A_K\in\{-1,1,*\}^{M\times N}$ by
\[
(A_K)_{i,j} =
\begin{cases}
+1 & \text{if } j^+\in F_i,\\ 
-1 & \text{if } j^-\in F_i,\\ 
*  & \text{otherwise.}
\end{cases}
\]
Since $K$ is free, each facet contains at most one of
$j^+,j^-$, hence $A_K$ is well-defined.
By construction the facets of $S(A_K)$ are precisely the
$F_i$ and their $\tau$–images.
Thus $K\cong S(A_K)$.
\end{proof}

\subsection{Sign-rank is a linear analog of $\Z_2$-index}\label{section:indlinsignrank}
Next we prove the main lemma that reformulates sign-rank as the linear analog of $\Z_2$-index.
For a free $\Z_2$-complex $K$ realized as $|K|\subset \R^N$, define
\[
\ind^{\mathrm{lin}}(K)
\coloneqq
\min \Bigl\{ d \ge 0 :
\exists\ \text{linear } g\colon\R^N \to \R^{d+1}
\text{ such that } 0 \notin g(|K|) \Bigr\}.
\]

 \mainlem*

Since the map $g$ induces a $\Z_2$–equivariant map $f\colon |K|\to \Sp^{d}$, defined by $f(x)\coloneq\frac{g(x)}{\|g(x)\|}$, we have that for any free $\Z_2$-complex $K$
\[
\ind(K) \le \ind^{\mathrm{lin}}(K)
\]
Hence, an immediate corollary of \Cref{lem:main-lemma} is as follows.
\mainineq*
\begin{proof}[Proof of \Cref{lem:main-lemma}] We prove both inequalities. 

\medskip

\textbf{Part 1: $\ind^{\mathrm{lin}}(S(A)) \leq \srank(A)-1.$}
Assume $\srank(A)\leq d$.
By definition, there exist nonzero vectors
\[
u_1,\dots,u_N\in \R^d\setminus\{0\},
\qquad
v_1,\dots,v_M\in \R^d\setminus\{0\}
\]
such that for every specified entry $A_{rj}\in\{\pm 1\}$ we have
\[
\sign(\langle u_r, v_j\rangle)=A_{rj}.
\]
Define a linear map
\(
g:\R^N\to \R^d
\)
by setting
\[
g(e_j)=v_j \qquad (j\in[N]).
\]
Since \(g\) is linear, we also have
\(
g(-e_j)=-v_j.
\)
We claim that $0\notin g(|S(A)|)$.
Note that $|S(A)|$ is the crosspolytope realization of $S(A)$.

Let \(x\in |S(A)|\). Then \(x\) lies in some simplex of \(S(A)\). Without loss of generality, there exists a row \(r\in[M]\) and disjoint sets \(J_r^+,J_r^-\subseteq [N]\) such that
\[
A_{rj}=1 \text{ for } j\in J_r^+,\qquad A_{rj}=-1 \text{ for } j\in J_r^-,
\]
and
\[
x=\sum_{j\in J_r^+}\alpha_j e_j-\sum_{j\in J_r^-}\beta_j e_j
\]
for some coefficients \(\alpha_j,\beta_j> 0\) satisfying
\(
\sum_{j\in J_r^+}\alpha_j+\sum_{j\in J_r^-}\beta_j=1.
\)
Therefore
\[
g(x)=\sum_{j\in J_r^+}\alpha_j v_j-\sum_{j\in J_r^-}\beta_j v_j
\]
and
\[
\langle u_r,g(x)\rangle
=
\sum_{j\in J_r^+}\alpha_j\langle u_r,v_j\rangle
-
\sum_{j\in J_r^-}\beta_j\langle u_r,v_j\rangle.
\]
For \(j\in J_r^+\), we have \(\langle u_r,v_j\rangle>0\), and for \(j\in J_r^-\), we have \(\langle u_r,v_j\rangle<0\). Hence every term on the right-hand side is strictly positive, therefore,
\(
\langle u_r,g(x)\rangle>0
\)
so $g(x)\neq 0$. By definition
\(
\ind^{\mathrm{lin}}(S(A))\le d - 1.
\)

\medskip
\textbf{Part 2: $\ind^{\mathrm{lin}}(S(A)) \geq \srank(A) -  1.$}
Assume \(\ind^{\mathrm{lin}}(S(A))\le d-1\). Then there exists a linear map
\(
g:\R^N\to \R^{d}
\)
such that
\(
0\notin g(|S(A)|).
\)
Set
\[
v_j:=g(e_j)\in \R^d \qquad (j\in[N]).
\]
Fix a row \(r\in[M]\), and define
\[
J_r^+:=\{j\in[N]: A_{rj}=1\},
\qquad
J_r^-:=\{j\in[N]: A_{rj}=-1\}.
\]
Note that
\begin{equation}\label{eq:emptyintersection}
\sigma_r = \conv\left(\{e_j:j\in J_r^+\}\cup \{-e_j:j\in J_r^-\}\right)
\end{equation}
is a simplex of $|S(A)|$ hence
\[0\notin C \coloneq g(\sigma_r) = \conv\left(\{v_j:j\in J_r^+\}\cup \{-v_j:j\in J_r^-\}\right). \]
Let $u_r$ be the unique closest point to 0 in $C$.
For any $x\in C$, the function $t\mapsto \|u_r+t(x-u_r)\|^2$ is minimized at $t=0$, giving $\langle u_r,x-u_r\rangle\ge 0$ and hence
\[\langle x,u_r\rangle \geq \|u_r\|^2_2 >0 \quad \text{for all } x\in C.\]
In particular 
\[
\langle u_r,v_j\rangle>0 \quad \text{for all } j\in J_r^+,
\qquad
\langle u_r,-v_j\rangle>0 \quad \text{for all } j\in J_r^-.
\]
Hence if \(A_{rj}\in\{\pm1\}\), then
\(
\sign(\langle u_r,v_j\rangle)=A_{rj}.
\)
Doing this for each row \(r\) gives a sign-rank realization of \(A\) in \(\R^d\).

\end{proof}

\section{Main Application:  Sign-rank of Gap Hamming Distance}\label{section:application_gaphamming}

Let $k < n/2$ and define $\GHD_k^n:\{0,1\}^n\times \{0,1\}^n\to \{-1,1,*\}$ by
\[
\GHD_{k}^n(x,y)=
\begin{cases}
1 & \text{if } d_H(x,y)\leq k,\\
-1 & \text{if } d_H(x,y)\geq n-k,\\
* & \text{otherwise}.
\end{cases}
\]
The following fact is essentially trivial.
\begin{fact}\label{fact:ghdupper}
    For any $k$ and $n$, $\srank(\GHD_{k}^n)\leq   2k+1$.
\end{fact}
\begin{proof}
Define
\[
u(x) := ((-1)^{x_1},\dots,(-1)^{x_{2k+1}})\in\{-1,1\}^{2k+1},\qquad
v(y) := ((-1)^{y_1},\dots,(-1)^{y_{2k+1}})\in\{-1,1\}^{2k+1}.
\]

\smallskip
Suppose $d_H(x,y)\leq k$.
Then the total number of disagreements between $u(x)$ and $v(y)$ is at most $k$, hence the total number of agreements is at least $k+1$. Therefore,
$\sign(\langle u(x),v(y)\rangle)=1$.
The case of $d_H(x,y)\geq n-k$ is similar.

\end{proof}
The main result of this section is to prove an almost sharp lower bound on $\srank(\GHD_{k}^n)$ for any choice of $k < n/2$.
\strongbound*


\subsection{Weaker bound via CW-complex of the hypercube}\label{section:weaker}
As discussed in \Cref{subsection:gap_hamming_intro}, to obtain a lower bound on the sign-rank, it suffices to get a lower bound for the coindex. We prove the weaker bound as a warm up.
\begin{theorem}[Weak bound]
    \label{thm:weak_lower_GHD}
    For any $k < n/2$, $\coind(S(\GHD_{k}^n))\geq k$.
\end{theorem}
The main idea is to give a $\Z_2$-equivariant map from the $k$-skeleton of the hypercube, denoted by $H_n^{(k)}$, which is known to be $(k-1)$-connected, to our complex $S(\GHD_{k}^n)$. 
This approach is not able to improve the bound beyond $k$ because the $k$-skeleton is not $k$-connected. 

Before proving \Cref{thm:weak_lower_GHD}, we need to discuss connectivity of the skeleton of the hypercube.
\paragraph{The $k$-skeleton of the hypercube.}
Let $H_n$ denote the boundary of the  CW-complex realization of the hypercube graph $Q_n$.
Denote by $H_n^{(r)}$ the $r$-skeleton of $H_n$. For the background on CW-complexes the reader is referred to \cite{matouvsek2003using}. Informally, facets of $H_n^{(r)}$ are all the $r$-dimensional faces of the $n$-dimensional hypercube.
Formally, let $I \subseteq [n]$ with $|I| = r$, and let $a \in \{0,1\}^{[n]\setminus I}$.
The corresponding $r$-dimensional face of the cube $[0,1]^n$ is
\[
F_{I,a}
=
\{ x \in [0,1]^n \;:\; x_i = a_i \text{ for all } i \notin I \}.
\]
which is homeomorphic to $[0,1]^r$. Also its vertex set is 
\(
F_{I,a} \cap \{0,1\}^n.
\)
We need the following lemma (see, e.g., \cite[Sec.~11]{Bjorner1995}).
\begin{lemma}\label{lem:rskeleton_hypercube}
$H_n^{(r)}$ is $(r-1)$-connected.
\end{lemma}

\medskip
\begin{proof}[Proof of \Cref{thm:weak_lower_GHD}]
For $x\in\{0,1\}^n$ define the Hamming ball
\[
B(x,k) \coloneqq \{\, y\in\{0,1\}^n : d_H(x,y)\le k \,\}.
\]
Let $\bar x$ denote the bitwise complement of $x$.
The complex $S(\GHD_k^n)$ is defined on the vertex set
\(
V^- \sqcup V^+,
\)
where $V^-,V^+$ are two disjoint copies of $\{0,1\}^n$.
By definition of $\GHD_k^n$, its maximal simplices are precisely the joins
\[
B(x,k)^- * B(\bar x,k)^+,
\qquad x\in\{0,1\}^n.
\]

   We are going to construct a $\Z_2$-equivariant map $h$ from $|H_n^{(k)}|$ to $|S(\GHD_k^n)|$, where $|H_n^{(k)}|$ is equipped with a flip map $x \mapsto \overline{x}$. This will finish the proof, because by \Cref{lem:rskeleton_hypercube} $H_n^{(k)}$ is $(k-1)$-connected, hence 
   \[
\coind\left(S(\GHD_k^n)\right)
\;\ge\;
\coind\left(H_n^{(k)}\right) 
\;\ge\;
k,
\]
where the last inequality is via \Cref{prop:coindex_vs_connectivity}.
\smallskip

\emph{ The map $h$.}
Define a $\Z_2$-map
\[
h : H_n^{(k)} \longrightarrow |S(\GHD_k^n)|
\]
on vertices by
\[
h(y)\coloneqq \tfrac12\, y^- + \tfrac12\, (\bar y)^+ \in |V^- * V^+|, 
\qquad y\in\{0,1\}^n,
\]
and extend linearly over each face of $H_n^{(k)}$ using any symmetric triangulation without additional vertices.  
For every vertex $y$ we have
\[
h(\bar y)= \tfrac12\, (\bar y)^- + \tfrac12\, y^+ = \tau\bigl(\tfrac12\, y^- + \tfrac12\,(\bar y)^+\bigr)=\tau(h(y)),
\]
so $h$ is $\Z_2$-equivariant.

\smallskip

It remains to check that $h$ lands in $|S(\GHD_k^n)|$.
Let $F$ be any face of $H_n^{(k)}$.
Choose a vertex $x_0\in F$.
Since $F$ is a face of dimension at most $k$ in the cube,
there exists a set $J\subseteq[n]$ with $|J|\le k$
such that only coordinates in $J$ vary on $F$.
Hence for every vertex $y\in F$,
\[
d_H(x_0,y)\le k,
\]
so $y\in B(x_0,k)$ and similarly
$\bar y\in B(\bar x_0,k)$.
Therefore, for every vertex $y\in F$,
\[
h(y)\in \bigl|B(x_0,k)^- * B(\bar x_0,k)^+\bigr|,
\]
which is a simplex of $S(\GHD_k^n)$.
Because $h$ is linear on $F$, we obtain
\[
h(F)\subseteq \bigl|B(x_0,k)^- * B(\bar x_0,k)^+\bigr|
\subseteq |S(\GHD_k^n)|.
\]
Thus $h$ is a well-defined $\Z_2$-map $H_n^{(k)}\to |S(\GHD_k^n)|$.
\end{proof}

\subsection{Strong bound via Vietoris--Rips complex of the hypercube}\label{section:stronger}
In this section, we obtain an essentially sharp bound on the $\coind(S(\GHD_k^n))$. 
\strongboundGHD*
To prove this result, we work with the Vietoris--Rips complex of the hypercube graph. We give a brief description as follows. The reader is referred to~\cite{EdelsbrunnerHarer2010, Ghrist2014ElementaryAppliedTopology, ChazalMichel2021IntroductionTDA} 
for more details.
\begin{definition}[Vietoris--Rips complex of the hypercube.]
    For $r \in \mathbb{N}$, the Vietoris--Rips complex $\VR(Q_n,k)$ is the simplicial complex on vertex set $\{0,1\}^n$ whose simplices are all 
    subsets of Hamming diameter at most~$k$. Equivalently, for a subset $\sigma \subseteq \{0,1\}^n$,
\[
\sigma \in \VR(Q_n,k)
\quad \Longleftrightarrow \quad
\forall x,y \in \sigma, \; d_H(x,y) \le k.
\]
\end{definition}

It is easy to see that the map $h$, given in \Cref{thm:weak_lower_GHD}, gives a $\Z_2$-equivariant map from $\VR(Q_n,k)$ to $S(\GHD_k^n)$. Hence we immediately get the following lemma.
\begin{lemma}\label{lem:sign_index_VR_index}
    For any $k < \frac{n}{2}$, 
    \[\coind(S(\GHD_{k}^n))\geq \coind(\VR(Q_n,k)).\]
\end{lemma}
\begin{proof}

Define a $\Z_2$-map
\[
h \colon |\VR(Q_n,k)|\longrightarrow |S(\GHD_{k}^n)|
\]
on vertices by
\[
h(y)\coloneqq \tfrac12\, y^- + \tfrac12\, (\bar y)^+ \in |V^- * V^+|, 
\qquad y\in\{0,1\}^n,
\]    
extend it linearly on each face, the same argument as in \Cref{thm:weak_lower_GHD} implies that such extension still lies in $|S(\GHD_{k}^n)|$. 
\end{proof}
Therefore the goal of the rest of this section is to prove an essentially sharp lower bound on $\coind(\VR(Q_n,k))$, which could be of independent interest.
\VRZtwo*
Note that the bound in \Cref{lem:VRZ2} is essentially sharp:
\[\coind(\VR(Q_n,k))\leq \ind(\VR(Q_n,k))\leq \srank(\GHD_k^n)-1\leq 2k.\]

The key advantage of the $\VR(Q_n,k)$ is that it reflects
the Hamming metric on $\{0,1\}^n$ rather than only the cubical face structure as captured by $H_n^{(k)}$.
As $k$ increases, in particular when $k$ is  larger than $\frac{n}{2}$ by more than a standard deviation, $k$-neighborhood of each vertex starts to cover all but small portion of the  hypercube, which implies strong connectivity properties of $\VR(Q_n,k)$.
In particular, the following recent result of Bendersky and Grbic \cite{bendersky2023connectivity} quantifies this property.
\begin{theorem}[\cite{bendersky2023connectivity}, Theorem 2.3]
\label{thm:bendersky}
Let $\alpha_{n,k} = \frac{2^{n-1}}{\sum\limits_{i=k+1}^n\binom{n}{i}}$. Then $\VR(Q_n,k)$ is $(\alpha_{n,k}-2)$-connected.
\end{theorem}
However, a limitation of the above is that the  value of $\alpha_{n,k}$ becomes large only when the upper
binomial tail $\sum_{i>k}\binom{n}{i}$ is small.
By   Chernoff bound, this happens precisely when
\[
k \ge \frac{n}{2} + \Theta\big(\sqrt{n\log n}\big).
\]
However, in \Cref{lem:VRZ2} we are working in the opposite regime where $k< \frac{n}{2}$ can be arbitrarily small and \Cref{thm:bendersky} does not directly give nontrivial
connectivity bounds for $\VR(Q_n,k)$ itself.
However, remember that we only need a lower bound on $\Z_2$-index of $\VR(Q_n,k)$. To do so,
instead of lower bounding the connectivity of $\VR(Q_n,k)$, we lower bound the connectivity of a carefully chosen $\Z_2$-equivariant subcomplex of $\VR(Q_n,k)$ which is sufficient to imply \Cref{lem:VRZ2}, since unlike connectivity, $\Z_2$-coindex is monotone. 
\begin{lemma}\label{lem:subcomplex_connected}
 Let $k<\frac n2$. Then there exists a $(1-o_k(1))2k$-connected
$\Z_2$-complex that admits a $\Z_2$-equivariant embedding into
$\VR(Q_n,k)$.
\end{lemma}

\subsection{Proof of Lemma \ref{lem:subcomplex_connected}}
The main idea is to exploit Theorem~\ref{thm:bendersky}
by applying it not in dimension $n$, but inside carefully chosen
lower-dimensional faces of the cube.
Then gluing together these highly connected pieces via the nerve lemma produces a highly connected $\Z_2$-equivariant subcomplex
of $\VR(Q_n,k)$.  

We  need the following as a corollary of \Cref{thm:bendersky}.
\begin{claim}\label{claim:large_r_connectivity}
    Let $k\geq 5$ and $t\ge 2$ be chosen such that 
\[k > \frac{t}{2} + 2\sqrt{t\log t}.\]
Then for any $t'\leq t$, $\VR(Q_{t'},k)$ is $(t-1)$-connected.
\end{claim}
\begin{proof}
Fix $t'\le t$.

\smallskip

\emph{ Case 1: $t'\le k$.}
Then every subset of $\{0,1\}^{t'}$ has Hamming diameter at most $k$, hence
$\VR(Q_{t'},k)$ is the full simplex on $2^{t'}$ vertices. In particular,
$\VR(Q_{t'},k)$ is contractible.

\smallskip
\emph{ Case 2: $t'> k$.}
Since $t'\le t$ and $k>\frac{t}{2}+2\sqrt{t\log t}$, we also have
\[
2k-t' \;\ge\; 2k-t \;>\; 4\sqrt{t\log t}.
\]
Applying  Chernoff bound gives
\[
\sum_{i=k+1}^{t'} \binom{t'}{i}
\le
2^{t'} \exp\left(-\frac{(2k-t')^2}{2t'}\right).
\]
Using $t'\le t$ and $2r-t'\ge 2k-t$ we obtain  
\[
\frac{(2k-t')^2}{2t'}
\;\ge\;
\frac{(2k-t)^2}{2t}
\;>\;
\frac{(4\sqrt{t\log t})^2}{2t}
=
8\log t,
\]
and hence
\[
\sum_{i=k+1}^{t'} \binom{t'}{i}
\le
2^{t'} e^{-8\log t}
=
2^{t'} t^{-8}.
\]
Therefore
\[
\alpha_{t',k}
=
\frac{2^{t'-1}}{\sum_{i=k+1}^{t'} \binom{t'}{i}}
\;\ge\;
\frac{2^{t'-1}}{2^{t'} t^{-8}}
=
\frac12\, t^{8}.
\]
In particular, for $t\ge 2$ we have $\alpha_{t',k}\ge t+1$
(e.g.\ since $\tfrac12 t^8-(t+1)\ge 0$ for $t\ge 2$), and thus by
Theorem~\ref{thm:bendersky} the complex $\VR(Q_{t'},k)$ is
$(\alpha_{t',k}-2)$-connected, hence at least $(t-1)$-connected.
\end{proof}

Returning to the proof of \Cref{lem:subcomplex_connected}, let $t$ be maximal such that
\begin{equation}
\label{eq:t_choice}
k > \frac{ t}{2} + 2\sqrt{ t\log  t}.
\end{equation}
Maximality of $ t$ implies that
\[
k \le \frac{ t+1}{2} + 2\sqrt{( t+1)\log( t+1)}.
\]
Rearranging gives
\[
 t \geq 2k - 1-4\sqrt{2k\log (2k)}
\geq  (1-o_k(1))2k.
\]
\noindent
\emph{ The subcomplex $\VR^t(Q_n,k)$.} 
Let $\VR^{t}(Q_n,k)$ denote the subcomplex of $\VR(Q_n,k)$ consisting of those simplices
whose vertices
\begin{enumerate}
    \item have diameter at most $k$, and
  \item that subset lies in some $t$-dimensional face of the hypercube $H_n$.
\end{enumerate}
In other words, each simplex of $\VR^{t}(Q_n,k)$ is of the form
\(
\sigma \;=\; \tau \cap F_{I,a},
\)
where $\tau$ is a simplex of $\VR(Q_n,k)$ and
\[
F_{I,a}
\;=\;
\{\, x \in \{0,1\}^n \;:\; x_j = a_j \text{ for all } j \notin I \,\}
\]
for some index set $I \subseteq [n]$ with $|I|=t$ and some fixed
$a \in \{0,1\}^{[n]\setminus I}$.
Thus, $\VR^{t}(Q_n,k)$ is obtained by restricting attention to simplices supported
on $t$-faces of the hypercube. Furthermore, observe that $\VR^{t}(Q_n,k)$ is a $\Z_2$-free subcomplex of $\VR(Q_n,k)$.
Indeed, the antipodal action $x \mapsto \bar x$ on $Q_n$ preserves Hamming
distance and sends $t$-dimensional faces to $t$-dimensional faces, so it preserves the defining conditions of $\VR^{t}(Q_n,k)$.

\medskip
In the rest of the proof we show that $\VR^{t}(Q_n,k)$ is $(t-1)$-connected. 

\paragraph{Connectivity of $\VR^{t}(Q_n,k)$.}
Connectivity is established in the following steps.
\begin{enumerate}
    \item Cover $\VR^{t}(Q_n,k)$ with a suitable family $\mathcal{C}$ and elements of $\mathcal{C}$ and their nonempty intersections are  $(t-1)$-connected.
    \item Show that the nerve of the cover $\mathcal{C}$ itself is $(t-1)$-connected. 
    \item Apply the connectivity version of the nerve lemma to conclude that  $\VR^{t}(Q_n,k)$ is $(t-1)$-connected.
\end{enumerate}
\medskip
\noindent
\emph{ Step 1.}
For each $t$-dimensional face $F=F_{I,a}$ of the cube, define
\[
U(F)
\;:=\;
\VR(Q_n,k)\big|_{F}.
\]
Let
\[
\mathcal C
\;=\;
\{\, U(F) \;:\; F \text{ is a $t$-face of } Q_n \,\}.
\]
By definition of $\VR^{t}(Q_n,k)$, $\mathcal{C}$ forms a cover:
\[
\VR^{t}(Q_n,k)
=
\bigcup_{F} U(F).
\]
Each $F$ is canonically isomorphic to $\{0,1\}^t$ by restricting to the free coordinates of $F$, and under this
identification
\[
U(F) \cong \VR(Q_t,k).
\]
Now consider a finite intersection
\[
U(F_0)\cap\cdots\cap U(F_s).
\]
If it is nonempty, then
\[
F_0\cap\cdots\cap F_s
=
F'
\]
is a (possibly lower-dimensional) face of $Q_n$
of dimension $t'\le t$, and
\[
U(F_0)\cap\cdots\cap U(F_s)
=
\VR(Q_n,k)\big|_{F'}
\cong
\VR(Q_{t'},k).
\]
The last complex is $(t-1)$-connected by \Cref{claim:large_r_connectivity}

\medskip
\noindent
\emph{ Step 2.}
Here we show that the nerve  $\mathcal N=\Nrv(\mathcal C)$ is $(t-1)$-connected.
A finite collection $\{F_0,\dots,F_s\}$ spans a simplex in $\mathcal N$
iff
\[
U(F_0)\cap\cdots\cap U(F_s)\neq\varnothing,
\]
which by Step~1 holds iff
\[
F_0\cap\cdots\cap F_s\neq\varnothing.
\]
which is true iff
\[
\bar F_0\cap\cdots\cap \bar F_s\neq\varnothing.
\]
where $\bar F\subseteq [0,1]^n$ is the face of the CW-complex of the geometric realization of hypercube $H_n$ with vertex set $F = \bar F\cap \{0,1\}^n$.
Define a cover of the $t$-skeleton $(H_n)^{(t)}$
\[
\mathcal C'
=
\{\, \bar F \;:\; \bar F \text{ is a $t$-face of } H_n \,\}.
\]
Since we know that $\Nrv(\mathcal C')$ is isomorphic to $\Nrv(\mathcal{C})$, it is enough to show that $\Nrv(\mathcal C')$ is $(t-1)$-connected. To show this, note that
\[
(H_n)^{(t)}
=
\bigcup_{\bar F\in\mathcal C'} \bar F.
\]
Each $\bar F$ is contractible, and any nonempty finite intersection
of $t$-faces is again a (lower-dimensional) face, hence contractible.
Thus $\mathcal C'$ is a \emph{good cover} of $(H_n)^{(t)}$, hence by the good cover nerve \Cref{lem:goodcover_nerve}
\[
\big|\Nrv(\mathcal C')\big|
\;\simeq\;
(H_n)^{(t)}.
\]
Finally by \Cref{lem:rskeleton_hypercube}, $(H_n)^{(t)}$, is $(t-1)$-connected.

\medskip
\noindent
\emph{ Step 3.}
We now apply the connectivity version of the Nerve Lemma, \Cref{lem:connectivity_nerve},
to the cover $\mathcal C$ of $\VR^{t}(Q_n,k)$.
We have established that 
\begin{itemize}
    \item From Step~1: 
every nonempty intersection of members
of $\mathcal C$ is $(t-1)$-connected.
\item From Step~2: the nerve $\Nrv(\mathcal C)$ is $(t-1)$-connected.
\end{itemize}
Therefore \Cref{lem:connectivity_nerve} implies that
\[
\VR^{t}(Q_n,k)
=
\bigcup_{\sigma\in\mathcal C} \sigma
\]
is $(t-1)$-connected.

\section{Basic properties of sign complex}\label{section:basic}
\subsection{Transposing the matrix}
Let $A\in\{-1,1,*\}^{M\times N}$ and write
\[
S(A)=\bigcup_{i=1}^M \bigl(K_i^+ \cup K_i^-\bigr),
\]
where $K_i^+$ and $K_i^-$ are the two the subcomplexes induced by row $i$.
Let $\mathcal N_A$ be the nerve of this cover.
Then we have the following.
\begin{claim}\label{clm:nerve_equals_transpose}
\(
\mathcal N_A \cong S(A^t).
\)
\end{claim}
\begin{proof}
  Vertices of the nerve correspond to the simplices $K_i^+$ and $K_i^-$, so we identify them with
$i^+$ and $i^-$ respectively. A face of $\mathcal N_A$ has the form $R^+ \sqcup T^-$ with $R,T\subseteq [M]$.
By definition of the nerve, this is a face iff
\begin{equation}\label{eq:intersectionnonempty}
\bigcap_{i\in R} K_i^+ \;\cap\; \bigcap_{j\in T} K_j^- \neq \emptyset.
\end{equation}
 A point lies in this intersection iff there exists a column $c\in[N]$ such that:
\[
A_{ic}=+1 \text{ for all } i\in R,
\qquad
A_{jc}=-1 \text{ for all } j\in T,
\]
or the same with all signs flipped (which corresponds to choosing $c^-$ instead of $c^+$).

Vertices of $S(A)$ are $c^+$ and $c^-$ for $c\in[N]$.
Fix $c\in[N]$. Then, by unpacking the definitions of $K_i^\pm$:
\[
c^+ \in K_i^+ \iff A_{ic}=+1,
\qquad
c^+ \in K_j^- \iff A_{jc}=-1,
\]
and similarly
\[
c^- \in K_i^+ \iff A_{ic}=-1,
\qquad
c^- \in K_j^- \iff A_{jc}=+1.
\]
Therefore the intersection in \eqref{eq:intersectionnonempty} is nonempty
if and only if there exists some column $c\in[N]$ such that either
\begin{equation}\label{eq:pattern_plus}
A_{ic}=+1 \ \forall i\in R
\quad\text{and}\quad
A_{jc}=-1 \ \forall j\in T,
\end{equation}
or the same condition with all signs flipped (corresponding to choosing $c^-$
instead of $c^+$).  

\medskip
Now consider $S(A^t)$.
Each column $c$ of $A$ becomes a row of $A^t$, producing a simplex whose positive and negative parts are
\[
\{i : A_{ic}=+1\}, \qquad \{i : A_{ic}=-1\}.
\]
Hence $R^+\sqcup T^-$ is a face of $S(A^t)$ iff there exists $c\in[N]$ such that
\[
A_{ic}=+1 \text{ for all } i\in R,
\qquad
A_{jc}=-1 \text{ for all } j\in T.
\]

\medskip
\noindent
This is exactly the same condition as above. Therefore
\[
R^+\sqcup T^- \in \mathcal N_A
\quad\Longleftrightarrow\quad
R   ^+\sqcup T^- \in S(A^t).
\]

\end{proof}



    

\subsection{VC dimension vs. $\Z_2$-coindex}
 In this section, we show that VC dimension of $A$ gives a lower bound on $\coind(A)$.
    Let $A\in\{-1,1,*\}^{M\times N}]$, such that for each pattern $P\in \{-1,1\}^S$, there is a row whose restriction to $S$ is $P$.

\begin{definition}
        For a free $\Z_2$-complex $K$, let $\omega_\Diamond(K)$ be the largest $k$ such that there is a copy of $\partial \Diamond_k$ as a free $\Z_2$-subcomplex of $K$. 
\end{definition}
It is easy to see that $\omega_\Diamond(S(A))$, is equal to the notion of \emph{dual sign-rank} which was introduced by Alon, Moran, and Yehudayoff \cite{alon2016sign}, as the largest number of columns that are `antipodally shattered'.

\begin{theorem}\label{thm:VCDiamond}
    We have
    \[ \frac{\omega_\Diamond(S(A))}{2} \leq \VC(A)\leq \omega_\Diamond(S(A)) \leq  \coind(S(A))+1.\]
\end{theorem}
\begin{proof}
First we show $\VC(A) \le \omega_\Diamond(S(A))$.
Let $S \subseteq [N]$ be shattered.
Consider the induced subcomplex of $S(A)$ on
\[
V_S := \{ j^+, j^- : j \in S \}.
\]
For each $P \in \{-1,1\}^S$, set
\[
F_P := \{ j^{P(j)} : j \in S \}.
\]
Since $S$ is shattered, for every $P$ there exists a row $i$ with
$M_{i|S} = P$, and hence $F_P \subseteq \sigma_i^+ \subseteq S(A)$.
Hence every choice of one vertex from each pair $\{j^+,j^-\}$ is a face of
the induced subcomplex on $V_S$ which is simplicially isomorphic to
$\partial \Diamond_{|S|}$ and is invariant under
$j^+ \leftrightarrow j^-$.

\medskip

Now we show $\omega_\Diamond(S(A)) \le 2\,\VC(A)$.
Suppose $S(A)$ contains a $\mathbb{Z}_2$-invariant
subcomplex $\partial \Diamond_d$.
Then there exists $S \subseteq [N]$, $|S|=d$, such that for every
$P \in \{-1,1\}^S$ the face
\[
F_P := \{ j^{P(j)} : j \in S \}
\]
lies in $S(A)$.
By construction of $S(A)$, for each $P$ there is a row whose restriction
to $S$ equals either $P$ or $-P$.
Hence at least one element from each antipodal pair $\{P,-P\}$ is realized,
so the number of distinct patterns on $S$ is at least $2^{d-1}$.
If $\VC(A) < d/2$, then by Sauer–Shelah lemma the number of patterns realizable
on $S$ is at most
\[
\sum_{i=0}^{\lfloor d/2 \rfloor - 1} \binom{d}{i}
< 2^{d-1},
\]
which is a contradiction.

Finally, to see $\omega_\Diamond(S(A)) \leq  \coind(S(A))+1$, let $\omega_\Diamond(S(A)) = d$ and note that $|\partial\Diamond_d|$ is $\mathbb Z_2$-homeomorphic to $\mathbb S^{d-1}$ with the antipodal action, hence 
\[
\coind(\partial\Diamond_d)=d-1
\]
and by monotonicity of coindex,
\[
d-1=\coind(\partial \Diamond_d) \leq \coind(S(A)).
\]

\end{proof}

\section{Separations}\label{section:separation}
The main goal of this section is to provide separation among parameters in the inequality
\[\VC(A)\leq \coind(S(A))\leq \ind(S(A))\leq \srank(A).\]
\subsection{Index vs. sign-rank for total matrices}\label{section:index-signrank-total-separation}
We work with two cases of total and partial matrices separately. 
For total matrices, we give a separation of $O(\log N)$-vs-$\Omega(N)$ by considering random matrices where a linear lower bound on sign-rank is known from the work of Alon et. al, \cite{alon1985geometrical}. For the well-known Hadamard matrix, we get a similar separation of $O(\log N)$-vs-$\Omega(\sqrt{N})$ based on the work of Forster \cite{forster2002linear}.

For partial matrices we obtain a stronger separation of $\ind(S(A))=1$ and $\srank(A)=\Omega(\sqrt{N})$ by considering a random labeling of the incidence matrix of the finite projection plane.

We develop a general technique to upper bound $\Z_2$-index, and apply it to two well-known families of matrices of high sign-rank: random matrices, and the Hadamard matrix.

Let $A\in\{\pm1,*\}^{M\times N}$ be a partial sign matrix with column set $[N]$.
For each row $i\in[M]$ define 
\[
R_i^+ \;:=\; \{\, j\in[N] : A_{ij}=+1 \,\},
\qquad
R_i^- \;:=\; \{\, j\in[N] : A_{ij}=-1 \,\}.
\]
Let
\[
\mathcal{R}(A)
\;:=\;
\{\, R_i^+,\ R_i^- : i\in[M] \,\}
\]
be the family of shore-parts of facets induced by $A$.
Define the \emph{intersection family} of $\mathcal{F}(A)$ by
\[
\mathcal{I}(A)
\;:=\;
\Bigl\{
\bigcap_{t=1}^k F_t \ :\ k\ge 1,\ F_t\in \mathcal{R}(A)
\Bigr\}
\setminus\{\emptyset\}.
\]
Let
\[
h(A)
\;:=\;
\max\Bigl\{
k:\ \exists S_1\subsetneq S_2\subsetneq\cdots\subsetneq S_k
\text{ with each }S_i\in \mathcal{I}(A)
\Bigr\}
\]
be the height (maximum strict chain length) of $\mathcal{I}(A)$ under inclusion.
The main result of this section is the following.
\begin{theorem}[Chain height bounds $\Z_2$-index]\label{thm:chainheight-index}
For any partial sign matrix $A$, \[\ind(S(A)) \le 2h(A)-1.\]
\end{theorem}
Later we apply the above to random matrices and the Hadamard matrix and show these matrices have logarithmic height. Before proving \Cref{thm:chainheight-index}, we need one lemma about $\Z_2$-equivariant maps between order complexes.
\begin{lemma}[A poset map induces a simplicial $\Z_2$-map on order complexes]
\label{lem:posetmapsimplicial}
Let \(P,Q\) be finite posets with \(\Z_2\)-actions $\tau$ and $\nu$ respectively. Let
\(\phi:P\to Q\) be order-preserving and \(\Z_2\)-equivariant. Then \(\phi\)
induces a simplicial \(\Z_2\)-equivariant map
\[
\phi:\Delta(P)\to \Delta(Q).
\]
\end{lemma}

\begin{proof}
A simplex of \(\Delta(P)\) is a chain
\[
x_0<x_1<\cdots <x_t
\]
in \(P\). Since \(\phi\) is order-preserving,
\[
\phi(x_0)\le \phi(x_1)\le \cdots \le \phi(x_t).
\]
After deleting repetitions, this becomes a chain in \(Q\), hence a simplex of
\(\Delta(Q)\). Therefore \(\phi\) induces a simplicial map
\(\Delta(\phi):\Delta(P)\to \Delta(Q)\).

Since \(\phi\) is \(\Z_2\)-equivariant, for every \(x\in P\) we have
\[
\phi(\tau(x))=\nu(\phi(x)),
\]
and therefore the induced simplicial map commutes with the \(\Z_2\)-actions.
So \(\Delta(\phi)\) is a simplicial \(\Z_2\)-equivariant map.
\end{proof}

\begin{proof}[Proof of theorem 
\ref{thm:chainheight-index}]
Let $P$ be the poset of nonempty faces of $S(A)$, ordered by inclusion. Then the order complex $\Delta(P)$ coincides with $\sd(S(A))$, and since the barycentric subdivision preserves $\Z_2$-index, it is enough to show that 
\[\ind(\Delta(P))\leq 2h(A)-1.\]
Let $\mathcal{G}$ be the family of simplices:
\[
\mathcal{G} = \big\{R_i^+ * R_i^-\mid i \in [N]\big\} \cup \big\{R_i^- * R_i^+\mid i \in [M]\big\}.
\]
Define map $\phi \colon P \to P$ by
\[
\phi(X) := \bigcap_{\substack{G \in \mathcal{G}\\ X\subseteq G}} G,
\]
where $X$ is an arbitrary face of $S(A)$. We have the following observations:
\begin{itemize}
    \item We know that $\mathcal{G}$ generates all simpices of $S(A)$, so $\phi$ is well-defined. 
    \item Family $\mathcal{G}$ is symmetric under exchanging $i^-$ with $i^+$, so $\phi$ preserves symmetry as well.
    \item Map $\phi$ is order-preserving. Indeed, if $X\subseteq X'$ then
$\{G\in\mathcal{G}:X'\subseteq F\}\subseteq \{G\in\mathcal{G}:X\subseteq G\}$, hence
$\phi(X)\subseteq \phi(X')$.
\end{itemize}
These observations, combined with \Cref{lem:posetmapsimplicial} imply that $\phi$ induces a $\Z_2$ -equivariant map between corresponding order complexes.
\[\phi \colon \Delta(P) \to \Delta(\phi(P))\]
hence 
\[
 \ind(\Delta(P)) \le \ind(\Delta(\phi(P))).
 \]
 So it remains to bound \(\ind(\Delta(\phi(P)))\). Since \(\ind(K)\le \dim(K)\)
for every free \(\Z_2\)-complex \(K\), it suffices to bound the dimension of
\(\Delta(\phi(P))\).
A simplex of \(\Delta(\phi(P))\) is a strict chain
\[
X_1\subsetneq X_2\subsetneq \cdots \subsetneq X_t
\qquad (X_i\in \phi(P)).
\]
Each \(X_i\) has the form
\[
X_i=A_i^+\cup B_i^-,
\]
with \(A_i,B_i\in \mathcal I(A)\). Because the positive and negative vertices
are disjoint,
\[
X_i\subseteq X_{i+1}
\quad\Longleftrightarrow\quad
A_i\subseteq A_{i+1}\ \text{and}\ B_i\subseteq B_{i+1}.
\]
Moreover, since the chain is strict, at each step at least one of these two
inclusions is strict. Hence along the whole chain, the \(A_i\)'s can strictly
increase at most \(h(A)-1\) times, and the \(B_i\)'s can strictly increase at
most \(h(A)-1\) times. Hence
\[
t-1\le (h(A)-1)+(h(A)-1)=2h(A)-2,
\]
so
\[
t\le 2h(A)-1.
\]
Therefore every simplex of \(\Delta(\phi(P))\) has at most \(2h(A)\) vertices,
and so
\[
\dim(\Delta(\phi(P)))\le 2h(A)-1.
\]


\end{proof}

\paragraph{Application I: Random matrices.} Next we apply  \Cref{thm:chainheight-index} to random matrices, and show with high probability, random matrices have logarithmic height.

\begin{theorem}
\label{thm:random-height}
Let $A\in\{\pm1\}^{N\times N}$ be a random sign matrix whose entries are
independent and uniform.
Then there exists an absolute constant $C>0$ such that with probability $1-o(1):$
\[
h(A)\le C\log_2 N
\]
which in particular implies
\[
\ind(S(A)) \le 2C\log_2 N.
\]
\end{theorem}
\begin{proof}

For $T\subseteq[N]$ and a sign pattern $s\in\{\pm1\}^T$ define the
\emph{cell}
\[
C(T,s)
:=
\{\, j\in[N] : A_{ij}=s(i)\ \text{for all } i\in T \,\}.
\]
Every element of $\mathcal{I}(A)$ equals $C(T,s)$ for some $(T,s)$.
Moreover, if
\[
C(T',s') \subsetneq C(T,s),
\]
then necessarily $T\subsetneq T'$.
Hence along any strict chain in $\mathcal{I}(A)$,
the parameter $|T|$ strictly decreases at each step. 

By \Cref{lem:largeT-cells-small} (applied with $\varepsilon=\frac12$),
there exists an absolute constant $C_1>0$ such that with probability $1-o(1)$:
\begin{equation}\label{eq:largeT-small}
\forall\, T \subseteq [N]\ \text{with}\ |T|\ge \tfrac{3}{2}\log_2 N,\ 
\forall\, s\in\{\pm1\}^{T},\quad
|C(T,s)| \le C_1 \log N.
\end{equation}
Fix an outcome matrix $A$ where \Cref{eq:largeT-small} holds.
Now consider a strict chain in $\mathcal{I}(A)$,
\[
S_1 \subsetneq S_2 \subsetneq \cdots \subsetneq S_M,
\qquad\text{where } S_j=C(T_j,s_j).
\]
As observed above,
\[
|T_1|>|T_2|>\cdots>|T_M|.
\]
We split the chain indices $j\in[M]$ into two ranges according to $|T_j|$.

\medskip
\noindent\textbf{Range I: $|T_j|\ge \frac32\log N$.}
By \Cref{eq:largeT-small}, every $S_j$ in this range has size at most
$C_1\log N$. Since the chain of $S_j$'s is strict, their sizes must increase by at least $1$
each step; hence the number of indices $j$ in Range~I is at most $C_1\log N$.

\medskip
\noindent\textbf{Range II: $ |T_j| < \frac32\log N$.}
In this  range, we use the fact that the parameter $|T_j|$ decreases by at least $1$ each
step.
Hence the number of indices $j$ in Range~II is at most
$\frac32\log N$.
Overall we get that the length of the chain, $m$, is at most 
\[\left(\frac{3}{2}+C_1\right)\log N.\]
By choosing $C=\frac{3}{2}+C_1$ and applying \Cref{thm:chainheight-index} we get that $\ind(S(A))\leq 2h(A)-1\leq 2C\log N$.
\end{proof}

\begin{lemma}\label{lem:largeT-cells-small}
Fix $\varepsilon>0$. There exists a constant $C_1=C_1(\varepsilon)>0$ such that
for a random matrix $A\in\{\pm1\}^{N\times N}$ with independent uniform entries,
with probability $1-o(1)$, we have that simultaneously for all
$T\subseteq[N]$ satisfying \(|T|\ \ge\ (1+\varepsilon)\log N\) and all $s\in\{\pm1\}^T$:
\[
|C(T,s)| \ \le\ C_1\log N.
\]
\end{lemma}
\begin{proof}
    Fix $T\subseteq[N]$ with $|T|=t$ and $s\in\{\pm1\}^T$.
Then
\[
|C(T,s)| \sim \mathrm{Bin}(N,2^{-t}).
\]
Let $L:=C_1\log_2 N$ where $C_1$ is chosen later. By a union bound over $L$-subsets of columns,
\[
\Pr\big[|C(T,s)|\ge L\big]
\le \binom{N}{L}(2^{-t})^{L}
\le \left(\frac{eN}{L}\right)^L 2^{-tL}.
\]
Now union bound over all $(T,s)$ with $|T|=t$:
there are $\binom{N}{t}2^t$ such pairs, hence
\[
\Pr\Big[\exists\,T,s:|T|=t,\ |C(T,s)|\ge L\Big]
\le \binom{N}{t}2^t \left(\frac{eN}{L}\right)^L 2^{-tL}.
\]
Using $\binom{N}{t}\le (en/t)^t$ and $t\ge (1+\varepsilon)\log_2 n$, we get
\[
\Pr\Big[\exists\,T,s:|T|=t,\ |C(T,s)|\ge L\Big]
\le \left(\frac{2eN}{t}\right)^t \left(\frac{eN}{L}\right)^L 2^{-tL}.
\]
Summing this bound over all integers $t\ge (1+\varepsilon)\log_2 n$,
and choosing $C_1=C_1(\varepsilon)$ sufficiently large,
the factor $2^{-tL}$ dominates the the other factors,
giving total probability $o(1)$.
\end{proof}

\paragraph{Application II: The Hadamard matrix.}
Let $N=2^n$ and let $H=H_{N\times N}$ be the Hadamard matrix indexed by $\F_2^n$, defined by
\[
H_{x,y} := (-1)^{\langle x,y\rangle},\qquad x,y\in \F_2^n,
\]
where $\langle\cdot,\cdot\rangle$ is the standard dot product over $\F_2$.
It is known since Forster~\cite{forster2002linear} that
\[
\srank(H)\ge \sqrt{N} = 2^{n/2}.
\]
Let $K=S(H)$ be the sign complex associated with $H$.

\begin{lemma}\label{lem:hadamard-height}
    Let $H_{N\times N}$ be the Hadamard matrix as above. Then $h(H) = n +1= \log N+1$.
\end{lemma}
\begin{proof}
    For each $x\in\F_2^n$, the row $H_{x,\cdot}$ defines the partition
\[
R_x=\{\, y\in\F_2^n : \langle x,y\rangle=0 \,\},
\qquad
R_x^c=\{\, y\in\F_2^n : \langle x,y\rangle=1 \,\}.
\]
Thus every member of $\mathcal{F}(H)$ is an affine hyperplane in $\F_2^n$.

\emph{Lower bound $h(H)\ge n+1$.}
Let $e_1,\dots,e_n$  be the standard basis of $\F_2^n$.
The chain 
\[\{0\}\subsetneq \langle e_1\rangle\subsetneq \langle e_1,e_2\rangle\subsetneq \cdots \subsetneq \langle e_1,\cdots,e_n\rangle\] has length $n+1$.

\emph{Upper bound $h(H)\le n+1$.}
Any set in $\mathcal{I}(H)$ is an affine subspace of $\F_2^n$.
Along any strict inclusion chain of nonempty affine subspaces, the dimension drops by at least $1$ at each step,
so the chain length is at most $n+1$.
Thus $h(H)\le n+1$.
\end{proof}
Combining \Cref{lem:hadamard-height} and \Cref{thm:chainheight-index} gives the following theorem.
\begin{theorem}\label{theorem:HadamardZ2}
Let $H$ be the $N\times N$ Hadamard matrix with $N=2^n$. Then
\[
n-1 \leq \ind(S(H))\le 2n+1.
\]
\end{theorem}
For the lower bound, note that, it is easy to see that the VC dimension of $H_{2^n\times 2^n}$ is $n$. Hence by \Cref{prop:vccoindex}, $n-1$ is a lower bound on $\ind(K)$.

\subsection{Index vs. sign-rank for partial matrices}\label{section:index_signrank_separation_partial}
In this section, we show how to construct a partial matrix $A\in \{-1,1,*\}^{N\times N}$ such that 
\[
\ind(S(A)) \le 1
\qquad\text{but}\qquad
\srank(A) \ge \Omega\left( \sqrt N  \right).
\]
The construction is randomized and is based on the incidence matrix of the finite projective plane. The lower bound on sign-rank is based on a counting argument using Warren's theorem from real algebraic geometry. This is similar to the lower bound arguments in \cite{alon1985geometrical,ben1998localization}.

\begin{theorem}[Warren\cite{warren1968lower}] 
Let $f_1,\dots,f_t$ be real polynomials of degree at most $k$ in $m$ variables each, where $t\geq m$. Then the number of distinct sign patterns
\[
(\sign(f_1(x)),\dots,\sign(f_t(x))) \in \{-1,0,1\}^t
\]
is at most $(\frac{c k t}{m})^m$ for some absolute constant $c$.
\end{theorem}

\paragraph{Construction.}
Let $q$ be a prime power and let $N = q^2 + q + 1$. Let $\mathrm{PG}(2,q)$ denote the finite projective plane of order $q$. Recall that: (1) there are $N$ points and $N$ lines, (2) each line contains exactly $q+1$ points, and each point lies on exactly $q+1$ lines,
(3) every two distinct lines intersect in exactly one point.

Let $P \in \{0,1\}^{N\times N}$ be the point--line incidence matrix, where rows correspond to lines and columns to points:
\[
P_{\ell,p} =
\begin{cases}
1 & \text{if } p \in \ell,\\
0 & \text{otherwise.}
\end{cases}
\]

We define a random partial sign matrix $A\in \{-1,1,*\}^{N\times N}$ as follows:
\begin{itemize}
    \item if $P_{\ell,p} = 0$, set $A_{\ell,p} = *$,
    \item if $P_{\ell,p} = 1$, assign $A_{\ell,p} \in \{\pm1\}$ independently and uniformly at random.
\end{itemize}

\begin{proposition}\label{prop:indexsignrankpartialseparation}
    Let $A$ be the random matrix constructed above. Then with probability $1-o(1)$,
\[
\ind(S(A)) \le 1
\qquad\text{but}\qquad
\srank(A) \ge \Omega\left(\sqrt N\right).
\]

\end{proposition}

\begin{lemma}
\label{lem:ind-bound-by-facet-intersection}
    Let $K$ be a free $\Z_2$-complex where every pair of facets intersects in at most one vertex. Then $\ind(K) \le 1$.
\end{lemma}

\begin{proof}
    Let $G$ be the vertex-facet incidence graph of~$K$. The $\Z_2$-action on~$K$ induces a $\Z_2$-action on~$G$, and this action on~$G$ is again free: Indeed, no face of $K$ is fixed by the freeness of the action, so no vertex of~$G$ is fixed. No edge is fixed either since the action respects the bipartition of~$G$.

    Every face~$\sigma$ of~$K$ of dimension at least one lies in a unique facet $F(\sigma)$ since facets intersect in at most one vertex. Define the simplicial map $f\colon \sd(K) \to G$ from the barycentric subdivision of~$K$ to~$G$ by $f(\sigma) = \sigma$ if $\sigma$ is a vertex and $f(\sigma) = F(\sigma)$ otherwise. To see that $f$ is a simplicial map let $\sigma_0 \subset \dots \subset \sigma_k$ be a chain of faces in~$K$. If $\sigma_0$ is not a vertex, then $f(\sigma_i) = F(\sigma_i) = F(\sigma_0)$, so the face $\{\sigma_0, \dots, \sigma_k\}$ of $\sd(K)$ is mapped to the vertex $F(\sigma_0)$ in~$G$. If $\sigma_0$ is a vertex and $k \ge 1$ (the case $k=0$ being trivial) then $F(\sigma_1) = \dots = F(\sigma_k)$ and the image of the face $\{\sigma_0, \dots, \sigma_k\}$ is the edge $\{\sigma_0, F(\sigma_1)\}$ of~$G$. 

    Clearly, the map $f$ is $\Z_2$-equivariant and $G$ has dimension one. Thus $\ind(K) \le \ind(G) \le \dim(G) \le 1$.
\end{proof}

\begin{proof}[Proof of Prop.~\ref{prop:indexsignrankpartialseparation}]
    The Index upper bound follows immediately from Lemma~\ref{lem:ind-bound-by-facet-intersection}.
For the sign-rank lower bound, fix $d\geq 1$. We bound the probability that 
\[
\Pr\left[\srank(A)\le d\right].
\]

Suppose $\srank(A)\le d$. Then there exist vectors $u_\ell, v_p \in \R^d$ such that for every $\{\pm 1\}$ entry $(\ell,p)$,
\[
A_{\ell,p} = \sign(\langle u_\ell, v_p\rangle).
\]

Hence $A_{\ell,p}$ is given by a degree-$2$ polynomial in the $2Nd$ variables $\{u_\ell\},\{v_p\}$. 
Let $t = N(q+1) = \Theta(N^{3/2})$ be the number of $\{\pm 1\}$ entries in $A$. 

By Warren's theorem, with $k = 2$ and $m = 2Nd$, the number of realizable sign patterns on the $\{\pm 1\}$ entries is at most
\[
\left(\frac{ckt}{m}\right)^m
=
\left(\frac{2ct}{2Nd}\right)^{2Nd}
=
\left(\frac{c(q+1)}{d}\right)^{2Nd}.
\]
provided $2Nd\leq t$.
On the other hand, the $\{\pm1\}$ entries of $A$ are independent and uniformly random. Hence
\[
\Pr[\srank(A)\le d] \le \frac{\left(\frac{c(q+1)}{d}\right)^{2Nd}}{2^t}
\le \exp\left( O\left(Nd \log\left(\frac{q}{d}\right)\right) - t \right).
\]
Since $t = \Theta(N^{3/2})$, this probability is $o(1)$ if $d=\gamma q$ for some small enough constant $\gamma$
\[
Nd \log\left(\frac{q}{d}\right)\leq N \gamma q \log\left(\frac{1}{\gamma}\right)  \ll t,
\]
Therefore, with high probability,
\[
\srank(A) \ge \Omega\left({\sqrt{N}}\right).
\]

\end{proof}

\subsection{Index vs. coindex}\label{section:indexcoindex_separation}
For the sake of completeness, we include the following separation between coindex and index.
The odd-dimensional real projective spaces $\RP^{2N-1}$ is the space of $1$-dimensional real subspaces of~$\mathbb C^N$, which carries a free $\Z/2$-action induced by multiplication with~$i$. Stolz~\cite{Stolz1989} determined the $\Z/2$-index of~$\RP^{2N-1}$ as $N+1$, $N+2$, or~$N+3$, depending on the remainder of $N$ modulo~$8$. The coindex of a $\Z_2$-space~$X$ space is bounded from above by the Stiefel--Whitney height, that is, the largest integer $N$ such that the $N$th power of the Stiefel--Whitney class of the double cover $X\to X/\Z_2$ is non-zero in mod-$2$ cohomology. The Stiefel--Whitney height of $\RP^{2N-1}$ with a free $\Z_2$-action is one; see Singh~\cite{Singh2011}.

On the other hand, Matsushita~\cite{Matsushita2017} constructs free $\Z_2$-spaces with coindex one and arbitrarily large Stiefel--Whitney height. The Stiefel--Whitney height in turn is a lower bound for the $\Z_2$-index. To summarize, denoting Stiefel--Whitney height of a $\Z_2$-space~$X$ by $\mathrm{swh}(X)$ all inequalities in
\[
    \coind(X) \le \mathrm{swh}(X) \le \ind(X)
\]
have arbitrarily large separation.


\subsection{VC dimension vs. coindex}\label{section:vccoindex_separation}
Here we prove the following stronger separation between VC dimension and coindex.
\VCcoindexsep*

Here we construct a simplicial complex $K$ on $2N$ vertices, such that $K$ does not contain any $\partial\Diamond_2$, yet the $\coind(K) = \Omega(\log^2 N)$. Then we can consider the partial matrix $A = A_K$ associated with $K$, and conclude that $\VC(A)=1$ but $\coind(S(A))=\Omega(\log^2 N)$.
The complex $K$ comes from state-of-the-art triangulations of the projective space with small number of vertices,  due to Adiprasito et. al.,\cite{adiprasito2022subexponential}.
\begin{theorem}[Efficient triangulations of $\RP^d$]
\label{thm:AAK-antipodal-spheres}
For every integer $d \ge 1$, there exists a triangulation of $\RP^d$ with $N$ vertices where 
    \[
    N \le \exp\bigl((\tfrac12 + o(1)) \sqrt{d \log d}\bigr).
    \]

\end{theorem}

The above  implies the following separation between $\mathrm{VC}$ and coindex.

\begin{proof}
Let $T$ be a triangulation of $\RP^d$ with $N$ vertices given by
\Cref{thm:AAK-antipodal-spheres}. Let
\[
\pi \colon \Sp^d \to \RP^d
\]
be the antipodal double cover. Lifting the triangulation $T$ along $\pi$
gives a simplicial complex $K$ together with a free involution
$\tau$ such that
\[
|K| \cong \Sp^d
\qquad\text{and}\qquad
|K|/\tau \cong |T| \cong \RP^d .
\]
Equivalently, $K$ is a free $\Z_2$--triangulation of $\Sp^d$ whose
quotient is~$T$. Since every vertex of $T$ has exactly two lifts, the complex
$K$ has exactly $2N$ vertices, arranged in $N$ antipodal pairs.

Now let $A:=A_{K}$ be the partial sign matrix associated to
$K$ by \Cref{lem:signcomp_universal}. Then
\[
S(A)\cong K.
\]

 \textbf{Coindex.}
Since $|K|$ is $\Z_2$--homeomorphic to $\Sp^d$ with the antipodal
action, we have
\[
\coind(S(A))
=
\coind(K)
=
\coind(\Sp^d)
=
d.
\]

\textbf{VC dimension.}
We claim that $K$ does not contain $\partial\Diamond_2$ as a free
$\Z_2$--subcomplex, in other words, 
\[
\omega_\Diamond(K)\le 1.
\]
Therefore, by the inequality proved earlier,
\[
\VC(A)\le \omega_\Diamond(S(A))=\omega_\Diamond(K)\le 1.
\]
Suppose $\partial\Diamond_2 \subseteq K$. Let
\[
v_1 \sim v_2 \sim v_3 \sim v_4 \sim v_1
\]
be the $4$-cycle with $\tau(v_i)=v_{i+2}$.
Set $x_i := \pi(v_i)\in T$ and let
\[
\gamma := (x_1,x_2,x_3,x_4,x_1)
\]
be the induced loop in $T$.

Since $\pi:|K|\to |T|$ is a $2$-fold cover,
the number of lifts of $\gamma$ is at most $2$ and each lift is determined by its starting point.
However, $\gamma$ gives four distinct lifts:
\[
(v_1,v_2,v_3,v_4,v_1 ),\quad
(v_3,v_2,v_1,v_4,v_3),\quad
(v_1,v_4,v_3,v_2,v_1),\quad
( v_3,v_4,v_1,v_2,v_3).
\]
Thus $\gamma$ admits $\ge 4$ lifts, a contradiction.
Hence $\partial\Diamond_2 \not\subseteq K$.
Finally, by \Cref{thm:AAK-antipodal-spheres},
\[
N \le \exp\bigl((\tfrac12+o(1))\sqrt{d\log d}\bigr).
\]
Taking logarithms and rearranging gives
\[
d=\Omega\left(\frac{\log^2 N}{\log\log N}\right)
\]
proving the theorem.
\end{proof}

\section{Concluding remarks and open questions}\label{section:conclusion}

Consider the chain of inequalities 

 \[   \coind(S(A))\leq \ind(S(A))\leq \ind^{\mathrm{lin}}(S(A))=\srank(A)-1.\]

A downside is that all of these parameters are very hard to compute for a given input matrix $A$. Indeed, it is known that checking whether sign-rank is 3 or not is complete for the existential theory of the reals \cite{mnev2006universality,basri2009visibility, bhangale2015complexity} which is a class including PSPACE. Also index and coindex are known to be hard \cite{matouvsek2003using}.

However, it is possible to insert more useful topological invariants between $\coind$ and $\ind$ that are efficiently computable. 
For example, the Stiefel--Whitney height of a $\Z_2$-complex~$K$, denoted by $\mathrm{swh}(K)$, also known as cohomological index satisfies  
\[\coind(K)\leq \mathrm{swh}(K)\leq \ind(K)\]
and we discussed in \Cref{section:indexcoindex_separation} that there are $K$ with $\coind(K)=1$ and arbitrarily large $\mathrm{swh}(K)$. It is known that $\mathrm{swh}(K)$ is efficiently computable in the number of faces of $K$ (see p. 106 of \cite{matouvsek2003using}) and therefore, it can be a useful computable lower bound on sign-rank of arbitrary matrices that is provably stronger than analytic lower bounds, as in the case of Gap Hamming Distance.

\paragraph{Upper bounds on $\Z_2$-index.}
We showed that in the case of $N$ by $N$ random matrices and the Hadamard matrix, $\Z_2$-index of the associated complex is at most $O(\log N)$.
We leave the following as a question, which if true, indicates that the sign complex is not able to prove a strong lower bound on sign-rank of square matrices, and one needs to construct significantly richer complexes than the sign complex to prove strong bounds.  We phrase the question more generally for strongly regular free $\Z_2$-CW complexes, which are a broader class than free $\Z_2$-simplicial complexes. In a strongly regular CW complex the intersection of any two faces is a (possibly empty) face. 
\begin{question}\label{question:z2upperbound}
Is it true that any finite strongly regular 
free $\mathbb Z_2$--CW-complex $K$ with $N$ vertices and $M$ inclusion maximal cells, satisfies 
\(
\ind(K)\le O(\,\log N\,\log M) ?
\)
\end{question}
 If true, it implies that for any $N$ by $N$ sign matrix $A$,  $\ind(S(A))\leq O(\log^2 N)$ since $S(A)$ has $2N$ vertices and at most $2N$ facets.

\Cref{question:z2upperbound} may be viewed as a topological generalization of the classical
Figiel--Lindenstrauss--Milman inequality for centrally symmetric polytopes
\cite{figiel1976dimension}.
\begin{theorem}[Figiel--Lindenstrauss--Milman theorem \cite{figiel1976dimension}]\label{thm:FLM}
There exists an absolute constant $c>0$ such that for every centrally symmetric
polytope $P\subset \mathbb R^d$ with $N$ vertices and $M$ facets,
\[
 d  \leq c\cdot \log N\cdot \log M
\]
\end{theorem} 
To see the connection, let \(K=\partial P\) be the boundary complex of a centrally
symmetric polytope \(P\subset \mathbb R^d\), equipped with the antipodal
\(\Z_2\)-action. Then \(K\) is a strongly regular \(\Z_2\)-CW complex with \(N\)
vertices and \(M\) facets, and
\[\coind(K)=\ind(K)=d-1.\]
Indeed radial projection from $\partial P$ to $\Sp^{d-1}$ gives a \(\Z_2\)-equivariant map
\(|K|=\partial P \to \mathbb S^{d-1}\).

Hence a positive answer to \Cref{question:z2upperbound} (even with $\coind$ instead of $\ind$) would recover \Cref{thm:FLM}. However, it is  significantly stronger, because boundary complexes of centrally
symmetric polytopes form a highly constrained subclass of strongly regular
free \(\Z_2\)-CW complexes with strong local structure.

\paragraph{Equivalence of index and sign-rank for total matrices.}
Regardless of the answer to \Cref{question:z2upperbound}, one may still hope for an equivalence between sign-rank and index. We have already showed that for partial matrices this is impossible, since in \Cref{prop:indexsignrankpartialseparation} we constructed a partial matrix with $\ind(S(A)) =1$ while $\srank(A)\geq \Omega(\sqrt{N})$. However, we only have an exponential separation in the case of total matrices and it is still possible that in this case, there is a qualitative equivalence between index and sign-rank. 

\TotalSeparationQuestion*

If the question above has a positive answer, a large number of fundamental problems in sign-rank literature can be entirely formulated in terms of $\Z_2$-index of free $\Z_2$-simplicial complexes. 
Some of these problems include bounding closure properties of sign-rank and equivalence to semi-algebraic complexity \cite{lovett2019notes,hatami2022lower}, the possibility of upper bounding sign-rank based on margin \cite{linial2007complexity,hatami2023borsuk}, and upper bounding sign-rank based on factorization norm \cite{hatami2022lower}.

\paragraph{Equivalence of coindex and VC dimension for total matrices.}
In the case of partial matrices, we showed how to obtain a strong separation between VC dimension and coindex, and we stated the following interesting question for total matrices which would imply learnability is a topological phenomenon for total concept classes.
\VCcoindexseptotal*
Here we discuss another implication a positive answer to the above question.
If the above question has a positive answer, combined with our lower bound on coindex of $\GHD_k^n$, it immediately implies the following, which is a strong form of a 
 major open problem on learnability of disambiguations of the concept class of half-spaces with margin
conjectured by Alon et. al., \cite{alon2022theory}
\begin{conjecture}\label{conj:disambig}
    There exists a function $g(k)\to\infty$ such that the following is true: for every completion of $\GHD_k^n$ to a total sign matrix $A$ by changing all the $*$ outputs of $\GHD_k^n$ arbitrarily to $-1$,$1$,
\[
\VC(A)\ge g(k).
\]
\end{conjecture}
Now we explain how \Cref{question:VC-coindex-separation-total} implies implies \Cref{conj:disambig} with the function 
\(g(k) = f^{-1}(k).\) Here we can assume $f$ is strictly increasing hence invertible.
Indeed, suppose the answer to \Cref{question:VC-coindex-separation-total} is positive and let $A$ be an arbitrary total completion of $\GHD_k^n$. Then $S(\GHD_k^n)$ is a $\Z_2$-subcomplex of $S(A)$. By combining \Cref{eq:coindbound} with monotonicity  of $\coind(\cdot)$, we have: 
\[
f(\VC(A))\geq \coind(S(A))\ge \coind\!\bigl(S(\GHD_k^n)\bigr)=k.
\]
The last inequality is using our weak bound on $\coind(S(\GHD_k^n))$.

In fact, in a  surprising result, recently Chornomaz, Moran, and Waknine \cite{chornomaz2025spherical} have defined a notion of \emph{spherical dimension},  as the $\mathbb Z_2$-coindex of a topological $\mathbb Z_2$-space associated with a total concept class on the sphere, and proved that bounding spherical dimension in terms of VC dimension is \emph{equivalent} to the corresponding disambiguation-type statement for half-spaces with margin.  It may be possible to similarly here show that \Cref{conj:disambig} implies \Cref{question:VC-coindex-separation-total}.

\bibliographystyle{alpha}
\bibliography{main}

@inproceedings{alon2016sign,
      title={Sign rank versus VC dimension},
      author={Alon, Noga and Moran, Shay and Yehudayoff, Amir},
      booktitle={Conference on Learning Theory},
      pages={47--80},
      year={2016},
      organization={PMLR}
    }

@inproceedings{hatami2023borsuk,
      title={A Borsuk-Ulam lower bound for sign-rank and its applications},
      author={Hatami, Hamed and Hosseini, Kaave and Meng, Xiang},
      booktitle={Proceedings of the 55th Annual ACM Symposium on Theory of Computing},
      pages={463--471},
      year={2023}
    }

@inproceedings{alon1985geometrical,
      title={Geometrical realization of set systems and probabilistic communication complexity},
      author={Alon, Noga and Frankl, Peter and Rodl, Vojtech},
      booktitle={26th Annual Symposium on Foundations of Computer Science (sfcs 1985)},
      pages={277--280},
      year={1985},
      organization={IEEE}
    }

@article{forster2002linear,
      title={A linear lower bound on the unbounded error probabilistic communication complexity},
      author={Forster, J{\"u}rgen},
      journal={Journal of Computer and System Sciences},
      volume={65},
      number={4},
      pages={612--625},
      year={2002},
      publisher={Elsevier}
    }

@inproceedings{alon2022theory,
      title={A theory of PAC learnability of partial concept classes},
      author={Alon, Noga and Hanneke, Steve and Holzman, Ron and Moran, Shay},
      booktitle={2021 IEEE 62nd Annual Symposium on Foundations of Computer Science (FOCS)},
      pages={658--671},
      year={2022},
      organization={IEEE}
    }

@article{adiprasito2022subexponential,
      title={A subexponential size triangulation of {$\mathbb{R}P^n$}},
      author={Adiprasito, Karim and Avvakumov, Sergey and Karasev, Roman},
      journal={Combinatorica},
      volume={42},
      number={1},
      pages={1--8},
      year={2022},
      publisher={Springer}
    }

@article{bendersky2023connectivity,
      title={On the Connectivity of the Vietoris-Rips Complex of a Hypercube Graph},
      author={Bendersky, Martin and Grbic, Jelena},
      journal={arXiv preprint arXiv:2311.06407},
      year={2023}
    }

@book{matouvsek2003using,
      title={Using the Borsuk-Ulam theorem: lectures on topological methods in combinatorics and geometry},
      author={Matou{\v{s}}ek, Ji{\v{r}}{\'\i}},
      year={2003},
      publisher={Springer}
    }

@book{kozlov2008combinatorial,
      title={Combinatorial algebraic topology},
      author={Kozlov, Dmitry},
      year={2008},
      publisher={Springer}
    }

@article{paturi1986probabilistic,
      title={Probabilistic communication complexity},
      author={Paturi, Ramamohan and Simon, Janos},
      journal={Journal of Computer and System Sciences},
      volume={33},
      number={1},
      pages={106--123},
      year={1986},
      publisher={Elsevier}
    }

@inproceedings{klivans2001learning,
      title={Learning DNF in time},
      author={Klivans, Adam R and Servedio, Rocco},
      booktitle={Proceedings of the thirty-third annual ACM symposium on Theory of computing},
      pages={258--265},
      year={2001}
    }

@article{sherstov2011unbounded,
      title={The unbounded-error communication complexity of symmetric functions},
      author={Sherstov, Alexander A},
      journal={Combinatorica},
      volume={31},
      number={5},
      pages={583--614},
      year={2011},
      publisher={Springer}
    }

@inproceedings{hatami2022lower,
      title={Lower bound methods for sign-rank and their limitations},
      author={Hatami, Hamed and Hatami, Pooya and Pires, William and Tao, Ran and Zhao, Rosie},
      booktitle={Approximation, Randomization, and Combinatorial Optimization. Algorithms and Techniques (APPROX/RANDOM 2022)},
      pages={22--1},
      year={2022},
      organization={Schloss Dagstuhl--Leibniz-Zentrum f{\"u}r Informatik}
    }

@article{razborov2010sign,
      title={The sign-rank of AC \^{}0},
      author={Razborov, Alexander A and Sherstov, Alexander A},
      journal={SIAM Journal on Computing},
      volume={39},
      number={5},
      pages={1833--1855},
      year={2010},
      publisher={SIAM}
    }

@article{goos2025sign,
      title={Sign-Rank of $ k $-Hamming Distance is Constant},
      author={G{\"o}{\"o}s, Mika and Harms, Nathaniel and Imbach, Valentin and Sokolov, Dmitry},
      journal={arXiv preprint arXiv:2506.12022},
      year={2025}
    }

@inproceedings{forster2001relations,
      title={Relations between communication complexity, linear arrangements, and computational complexity},
      author={Forster, J{\"u}rgen and Krause, Matthias and Lokam, Satyanarayana V and Mubarakzjanov, Rustam and Schmitt, Niels and Simon, Hans Ulrich},
      booktitle={International Conference on Foundations of Software Technology and Theoretical Computer Science},
      pages={171--182},
      year={2001},
      organization={Springer}
    }

@article{borsuk1933drei,
  title={Drei S{\"a}tze {\"u}ber die n-dimensionale euklidische Sph{\"a}re},
  author={Borsuk, Karol},
  journal={Fundamenta Mathematicae},
  volume={20},
  number={1},
  pages={177--190},
  year={1933},
  publisher={Polska Akademia Nauk. Instytut Matematyczny PAN}
}

@article{fadell1988ideal,
  title={An ideal-valued cohomological index theory with applications to Borsuk-Ulam and Bourgin-Yang theorems},
  author={Fadell, Edward and Husseini, Sufian},
  journal={Ergodic Theory Dynam. Systems},
  volume={8},
  number={Charles Conley Memorial Issue},
  pages={73--85},
  year={1988}
}

@article{forman1998morse,
  title={Morse theory for cell complexes},
  author={Forman, Robin},
  journal={Advances in mathematics},
  volume={134},
  number={1},
  pages={90--145},
  year={1998},
  publisher={Academic Press}
}

@article{lovasz1978kneser,
  title={Kneser's conjecture, chromatic number, and homotopy},
  author={Lov{\'a}sz, L{\'a}szl{\'o}},
  journal={Journal of Combinatorial Theory, Series A},
  volume={25},
  number={3},
  pages={319--324},
  year={1978},
  publisher={Elsevier}
}

@incollection{blagojevic2017beyond,
  title={Beyond the Borsuk--Ulam theorem: the topological Tverberg story},
  author={Blagojevi{\'c}, Pavle VM and Ziegler, G{\"u}nter M},
  booktitle={A Journey Through Discrete Mathematics: A Tribute to Ji{\v{r}}{\'\i} Matou{\v{s}}ek},
  pages={273--341},
  year={2017},
  publisher={Springer}
}

@article{linial2007complexity,
  title={Complexity measures of sign matrices},
  author={Linial, Nati and Mendelson, Shahar and Schechtman, Gideon and Shraibman, Adi},
  journal={Combinatorica},
  volume={27},
  number={4},
  pages={439--463},
  year={2007},
  publisher={Springer}
}

@article{chornomaz2025spherical,
  title={Spherical dimension},
  author={Chornomaz, Bogdan and Moran, Shay and Waknine, Tom},
  journal={arXiv preprint arXiv:2503.10240},
  year={2025}
}

@article{blondal2025borsuk,
  title={Borsuk-Ulam and replicable learning of large-margin halfspaces},
  author={Blondal, Ari and Hatami, Hamed and Hatami, Pooya and Lalov, Chavdar and Tretiak, Sivan},
  journal={arXiv preprint arXiv:2503.15294},
  year={2025}
}

@article{blondal2025simplicial,
  title={Simplicial covering dimension of extremal concept classes},
  author={Blondal, Ari and Hatami, Hamed and Hatami, Pooya and Lalov, Chavdar and Tretiak, Sivan},
  journal={arXiv preprint arXiv:2511.11819},
  year={2025}
}

@inproceedings{chase2024local,
  title={Local borsuk-ulam, stability, and replicability},
  author={Chase, Zachary and Chornomaz, Bogdan and Moran, Shay and Yehudayoff, Amir},
  booktitle={Proceedings of the 56th Annual ACM Symposium on Theory of Computing},
  pages={1769--1780},
  year={2024}
}

@book{de2013course,
  title={A course in topological combinatorics},
  author={De Longueville, Mark},
  year={2013},
  publisher={Springer Science \& Business Media}
}

@incollection{Bjorner1995,
  author    = {Anders Bj{\"o}rner},
  title     = {Topological Methods},
  booktitle = {Handbook of Combinatorics},
  editor    = {R. L. Graham and M. Gr{\"o}tschel and L. Lov{\'a}sz},
  volume    = {2},
  pages     = {1819--1872},
  publisher = {Elsevier},
  address   = {Amsterdam},
  year      = {1995}
}

@article{ben1998localization,
  title={Localization vs. identification of semi-algebraic sets},
  author={Ben-David, Shai and Lindenbaum, Michael},
  journal={Machine Learning},
  volume={32},
  number={3},
  pages={207--224},
  year={1998},
  publisher={Springer}
}

@article{warren1968lower,
  title={Lower bounds for approximation by nonlinear manifolds},
  author={Warren, Hugh E},
  journal={Transactions of the American Mathematical Society},
  volume={133},
  number={1},
  pages={167--178},
  year={1968},
  publisher={JSTOR}
}

@article{Stolz1989,
  author  = {Stephan Stolz},
  title   = {The level of real projective spaces},
  journal = {Commentarii Mathematici Helvetici},
  volume  = {64},
  number  = {4},
  pages   = {661--674},
  year    = {1989},
  doi     = {10.1007/BF02564700}
}

@article{Matsushita2017,
  author  = {Takahiro Matsushita},
  title   = {Some examples of non-tidy spaces},
  journal = {Mathematical Journal of Okayama University},
  volume  = {59},
  number  = {1},
  pages   = {21--25},
  year    = {2017}
}

@article{Singh2011,
  author  = {Mahender Singh},
  title   = {Parametrized Borsuk-Ulam problem for projective space bundles},
  journal = {Fundamenta Mathematicae},
  volume  = {211},
  number  = {2},
  pages   = {135--147},
  year    = {2011}
}

@article{figiel1976dimension,
  title={The dimension of almost spherical sections of convex bodies},
  author={Figiel, T and Lindenstrauss, J and Milman, VD},
  year={1976}
}

@article{forster2006smallest,
  title={On the smallest possible dimension and the largest possible margin of linear arrangements representing given concept classes},
  author={Forster, J{\"u}rgen and Simon, Hans Ulrich},
  journal={Theoretical Computer Science},
  volume={350},
  number={1},
  pages={40--48},
  year={2006},
  publisher={Elsevier}
}

@article{alon2005crossing,
  title={Crossing patterns of semi-algebraic sets},
  author={Alon, Noga and Pach, J{\'a}nos and Pinchasi, Rom and Radoi{\v{c}}i{\'c}, Rado{\v{s}} and Sharir, Micha},
  journal={Journal of Combinatorial Theory, Series A},
  volume={111},
  number={2},
  pages={310--326},
  year={2005},
  publisher={Elsevier}
}

@article{frankl1981short,
  title={A short proof for a theorem of Harper about Hamming-spheres},
  author={Frankl, Peter and F{\"u}redi, Zolt{\'a}n},
  journal={Discrete Mathematics},
  volume={34},
  number={3},
  pages={311--313},
  year={1981},
  publisher={North-Holland}
}

@inproceedings{papakonstantinou2014overlays,
  title={Overlays and limited memory communication},
  author={Papakonstantinou, Periklis and Scheder, Dominik and Song, Hao},
  booktitle={2014 IEEE 29th Conference on Computational Complexity (CCC)},
  pages={298--308},
  year={2014},
  organization={IEEE}
}

@inproceedings{chattopadhyay2019equality,
  title={Equality alone does not simulate randomness},
  author={Chattopadhyay, Arkadev and Lovett, Shachar and Vinyals, Marc},
  booktitle={34th Computational Complexity Conference (CCC 2019)},
  pages={14--1},
  year={2019},
  organization={Schloss Dagstuhl--Leibniz-Zentrum f{\"u}r Informatik}
}

@article{sherstov2012communication,
  title={The communication complexity of gap hamming distance},
  author={Sherstov, Alexander A},
  journal={Theory of Computing},
  volume={8},
  number={1},
  pages={197--208},
  year={2012},
  publisher={Theory of Computing Exchange}
}

@article{gavinsky2025unambiguous,
  title={Unambiguous Parity-Query Complexity},
  author={Gavinsky, Dmytro},
  journal={Random Structures \& Algorithms},
  volume={66},
  number={3},
  pages={e70010},
  year={2025},
  publisher={Wiley Online Library}
}

@inproceedings{chakrabarti2011optimal,
  title={An optimal lower bound on the communication complexity of gap-hamming-distance},
  author={Chakrabarti, Amit and Regev, Oded},
  booktitle={Proceedings of the forty-third annual ACM symposium on Theory of computing},
  pages={51--60},
  year={2011}
}

@inproceedings{indyk2003tight,
  title={Tight lower bounds for the distinct elements problem},
  author={Indyk, Piotr and Woodruff, David},
  booktitle={44th Annual IEEE Symposium on Foundations of Computer Science, 2003. Proceedings.},
  pages={283--288},
  year={2003},
  organization={IEEE}
}

@article{jayram2013optimal,
  title={Optimal bounds for Johnson-Lindenstrauss transforms and streaming problems with subconstant error},
  author={Jayram, Thathachar S and Woodruff, David P},
  journal={ACM Transactions on Algorithms (TALG)},
  volume={9},
  number={3},
  pages={1--17},
  year={2013},
  publisher={ACM New York, NY, USA}
}

@inproceedings{brody2010better,
  title={Better gap-hamming lower bounds via better round elimination},
  author={Brody, Joshua and Chakrabarti, Amit and Regev, Oded and Vidick, Thomas and De Wolf, Ronald},
  booktitle={International Workshop on Randomization and Approximation Techniques in Computer Science},
  pages={476--489},
  year={2010},
  organization={Springer}
}

@article{Vietoris1927,
  author  = {Vietoris, Leopold},
  title   = {{\"U}ber den h\"oheren {Z}usammenhang kompakter {R}\"aume und eine {K}lasse von zusammenhangstreuen {A}bbildungen},
  journal = {Mathematische Annalen},
  volume  = {97},
  pages   = {454--472},
  year    = {1927}
}

@incollection{Gromov1987,
  author    = {Gromov, Mikhail},
  title     = {Hyperbolic groups},
  booktitle = {Essays in Group Theory},
  editor    = {Gersten, Stephen M.},
  series    = {Mathematical Sciences Research Institute Publications},
  volume    = {8},
  pages     = {75--263},
  publisher = {Springer},
  address   = {New York},
  year      = {1987},
  doi       = {10.1007/978-1-4613-9586-7_3}
}

@book{EdelsbrunnerHarer2010,
  author    = {Edelsbrunner, Herbert and Harer, John L.},
  title     = {Computational Topology: An Introduction},
  publisher = {American Mathematical Society},
  address   = {Providence, RI},
  year      = {2010},
  doi       = {10.1090/mbk/069}
}

@article{Carlsson2009,
  author  = {Carlsson, Gunnar},
  title   = {Topology and data},
  journal = {Bulletin of the American Mathematical Society},
  volume  = {46},
  number  = {2},
  pages   = {255--308},
  year    = {2009}
}

@article{AdamaszekAdams2022,
  author  = {Adamaszek, Micha{\l} and Adams, Henry},
  title   = {On {V}ietoris--{R}ips complexes of hypercube graphs},
  journal = {Journal of Applied and Computational Topology},
  volume  = {6},
  pages   = {177--192},
  year    = {2022},
  doi     = {10.1007/s41468-021-00083-1}
}

@article{Shukla2023,
  author  = {Shukla, Samir},
  title   = {On {V}ietoris--{R}ips Complexes (with Scale 3) of Hypercube Graphs},
  journal = {SIAM Journal on Discrete Mathematics},
  volume  = {37},
  number  = {3},
  pages   = {1472--1495},
  year    = {2023},
  doi     = {10.1137/22M1481440}
}

@article{Feng2025,
  author  = {Feng, Ziqin},
  title   = {Homotopy types of {V}ietoris--{R}ips complexes of hypercube graphs},
  journal = {Journal of Topology and Analysis},
  year    = {2025},
  doi     = {10.1142/S1793525325500062}
}

@article{AdamsVirk2024,
  author  = {Adams, Henry and Virk, {\v Z}iga},
  title   = {Lower Bounds on the Homology of {V}ietoris--{R}ips Complexes of Hypercube Graphs},
  journal = {Bulletin of the Malaysian Mathematical Sciences Society},
  volume  = {47},
  year    = {2024},
  doi     = {10.1007/s40840-024-01663-x}
}

@article{frankl1987forbidden,
  title={Forbidden intersections},
  author={Frankl, Peter and R{\"o}dl, Vojt{\v{e}}ch},
  journal={Transactions of the American Mathematical Society},
  volume={300},
  number={1},
  pages={259--286},
  year={1987}
}

@article{ghrist2008barcodes,
  title={Barcodes: the persistent topology of data},
  author={Ghrist, Robert},
  journal={Bulletin of the American Mathematical Society},
  volume={45},
  number={1},
  pages={61--75},
  year={2008}
}

@book{Ghrist2014ElementaryAppliedTopology,
  author    = {Robert Ghrist},
  title     = {Elementary Applied Topology},
  edition   = {1.0},
  publisher = {CreateSpace},
  year      = {2014},
  isbn      = {9781502880857}
}

@article{ChazalMichel2021IntroductionTDA,
  author  = {Fr{\'e}d{\'e}ric Chazal and Bertrand Michel},
  title   = {An Introduction to Topological Data Analysis: Fundamental and Practical Aspects for Data Scientists},
  journal = {Frontiers in Artificial Intelligence},
  volume  = {4},
  pages   = {667963},
  year    = {2021},
  doi     = {10.3389/frai.2021.667963}
}

@article{AdamsBushVirk2025ConnectivitySpheres,
  author  = {Henry Adams and Johnathan Bush and {\v Z}iga Virk},
  title   = {The Connectivity of {V}ietoris--{R}ips Complexes of Spheres},
  journal = {Journal of Applied and Computational Topology},
  volume  = {9},
  pages   = {15},
  year    = {2025},
  doi     = {10.1007/s41468-025-00214-y}
}

@misc{hatami2023personal,
  author = {Hatami, Hamed and Hatami, Pooya and Harms, Nathan and Hosseini, Kaave and Wang, Pushen},
  title = {Personal communication},
  year = {2023},
  note = {Personal communication, April 2023}
}

@article{bhangale2015complexity,
  title={The complexity of computing the minimum rank of a sign pattern matrix},
  author={Bhangale, Amey and Kopparty, Swastik},
  journal={arXiv preprint arXiv:1503.04486},
  year={2015}
}

@inproceedings{mnev2006universality,
  title={The universality theorems on the classification problem of configuration varieties and convex polytopes varieties},
  author={Mn{\"e}v, Nikolai E},
  booktitle={Topology and geometry—Rohlin seminar},
  pages={527--543},
  year={2006},
  organization={Springer}
}

@inproceedings{basri2009visibility,
  title={Visibility constraints on features of 3D objects},
  author={Basri, Ronen and Felzenszwalb, Pedro F and Girshick, Ross B and Jacobs, David W and Klivans, Caroline J},
  booktitle={2009 IEEE Conference on Computer Vision and Pattern Recognition},
  pages={1231--1238},
  year={2009},
  organization={IEEE}
}

@misc{lovett2019notes,
  author = {Shachar Lovett},
  title = {Communication Complexity Notes},
  year = {2019},
  note = {Lecture notes, UCSD}
}

@article{bjorner2003nerves,
  title={Nerves, fibers and homotopy groups},
  author={Bj{\"o}rner, Anders},
  journal={Journal of Combinatorial Theory, Series A},
  volume={102},
  number={1},
  pages={88--93},
  year={2003},
  publisher={Elsevier}
}

\end{document}